\documentclass[8,reqno]{amsart}
\usepackage{amsfonts}
\usepackage{amssymb}
\usepackage{a4wide}
\usepackage{pgf,pgfarrows,pgfautomata,pgfheaps,pgfnodes,pgfshade}
\usepackage{url}
\usepackage{lineno}
\usepackage{graphicx}
\usepackage[all]{xy}

\usepackage{hyperref}
\numberwithin{equation}{section}

\newtheorem{defi}{Definition}[section]
\newtheorem{thm}[defi]{Theorem}
\newtheorem{lemm}[defi]{Lemma}
\newtheorem{rem}[defi]{Remark}
\newtheorem{cor}[defi]{Corollary}
\newtheorem{prop}[defi]{Proposition}

\newcommand{\tr}{\mathrm{Tr}} 

\newcommand{\R}{\mathbb{R}}

\begin{document}

\title[Construction of Malliavin differentiable strong solutions]
{Construction of Malliavin differentiable strong solutions of SDEs under an integrability condition on the drift without the Yamada-Watanabe principle}
\author[D. R. Ba\~{n}os]{David R. Ba\~{n}os}
\address{CMA, Department of Mathematics, University of Oslo, Moltke Moes vei 35, P.O. Box 1053 Blindern, 0316 Oslo, Norway.}
\email{davidru@math.uio.no}
\author[S. Duedahl]{Sindre Duedahl}
\address{CMA, Department of Mathematics, University of Oslo, Moltke Moes vei 35, P.O. Box 1053 Blindern, 0316 Oslo, Norway.}
\email{sindred@math.uio.no}
\author[T. Meyer-Brandis]{Thilo Meyer-Brandis}
\address{Department of Mathematics,LMU, Theresienstr. 39, D-80333 Munich, Germany}
\email{Meyer-Brandis@math.lmu.de}
\author[F. Proske]{Frank Proske}
\address{CMA, Department of Mathematics, University of Oslo, Moltke Moes vei 35, P.O. Box 1053 Blindern, 0316 Oslo, Norway.}
\email{proske@math.uio.no}

\maketitle


\begin{abstract}
In this paper we aim at employing a compactness criterion of Da Prato, Malliavin, Nualart \cite{DPMN92} for square integrable Brownian functionals to construct unique strong solutions of SDE's under an integrability condition on the drift coefficient. The obtained solutions turn out to be Malliavin differentiable and are used to derive a Bismut-Elworthy-Li formula for solutions of the Kolmogorov equation.
\end{abstract}
 

 

\section{Introduction}

The object of study of this paper is the stochastic differential equation (SDE)
\begin{align}\label{introSDE}
X_t = x+ \int_0^t b(s,X_s^x)ds+B_t, \quad 0\leq t\leq T, \quad x\in \R^d,
\end{align}
where $B_{\cdot}$ is a $d$-dimensional Brownian motion on some complete probability space $(\Omega,\mathcal{F},\mu)$ with respect to a $\mu$-completed Brownian filtration $\{\mathcal{F}_t\}_{0\leq t\leq T}$ and where $b:[0,T]\times \R^d\rightarrow \R^d$ is a Borel-measurable function.

In this article we are interested in the analysis of strong solutions $X_{\cdot}$ of the SDE \eqref{introSDE}, that is an $\{\mathcal{F}_t\}_{0\leq t\leq T}$-adapted solution processes on $(\Omega, \mathcal{F},\mu)$ when the drift coefficient is irregular, e.g. non-Lipschitzian or discontinuous.

A widely used construction method for strong solutions in this case in the literature is based on the so-called Yamada-Watanabe principle. Using this principle, a once constructed weak solution, that is a solution which is not necessarily a functional of the driving noise, combined with pathwise uniqueness gives a unique strong solution. So
\begin{align}\label{Yamada-Watanabe}
\framebox[1.1\width]{Weak solution} \quad + \quad \framebox[1.1\width]{Pathwise uniqueness} \quad \Rightarrow \quad \framebox[1.1\width]{Unique strong solution}.
\end{align}

Here, pathwise uniqueness means the following: If $X_{\cdot}^{(1)}$ and $X_{\cdot}^{(2)}$ are $\{\mathcal{F}_t^{(1)}\}_{0\leq t\leq T}$- and respectively $\{\mathcal{F}_t^{(2)}\}_{0\leq t\leq T}$-adapted weak solutions on a probability space, then these solutions must coincide a.s. See \cite{YW71}. In the milestone paper from 1974 \cite{Zvon74}, A.K. Zvonkin used the Yamada-Watanabe principle in the one-dimensional case in connection with PDE techniques to construct a unique strong solution to \eqref{introSDE}, when $b$ is merely bounded and measurable. Subsequently, the latter result was generalised by A.Y. Veretennikov \cite{Ver79} to the multidimensional case.

Important other and more recent results in this direction are e.g. \cite{KR05}, \cite{GyK96} and \cite{Kryl01}. See also the striking work \cite{DPFPR13} in the Hilbert space setting, where the authors use solutions of infinite-dimensional Kolmogorov equations to obtain unique strong solutions of stochastic evolution equations with bounded and measurable drift for a.e. initial values.

In this article we want to employ a construction principle for strong solutions developed in \cite{MBP10}. This method which relies on a compactness criterion from Malliavin Calculus for square integrable functionals of the Brownian motion \cite{DPMN92} is in diametrical opposition to the Yamada-Watanabe principle \eqref{Yamada-Watanabe} in the sense that
\begin{align*}
\framebox[1.1\width]{Strong existence} \quad + \quad \framebox[1.1\width]{Uniqueness in law} \quad \Rightarrow \quad \framebox[1.1\width]{Strong uniqueness},
\end{align*}
that is the existence of a strong solution to \eqref{introSDE} and uniqueness in law of solutions imply the existence of a unique strong solution. A crucial consequence of this approach is the additional insight that the constructed solutions are regular in the sense of Malliavin differentiability.

We mention that this method has been recently applied in a series of other papers. See e.g. \cite{MMNPZ10}, where the authors obtain Malliavin differentiable solutions when the drift coefficient in $\R^d$ is bounded and measurable. Other applications pertain to the stochastic transport equation with singular coefficients \cite{MNP}, \cite{Nils14} or stochastic evolution equations in Hilbert spaces with bounded H\"{o}lder-continuous drift \cite{FNP14}. See also \cite{HP14} in the case of truncated $\alpha$-stable processes as driving noise and \cite{BNP15} in the case of fractional Brownian motion for Hurst parameter $H<1/2$, which is a non-Markovian driving noise.

Using the above mentioned new approach, one of the objectives of this paper is to construct Malliavin differentiable unique strong solutions to \eqref{introSDE} under the integrability condition
\begin{align}\label{introb}
b\in L^q([0,T], L^p(\R^d,\R^d))
\end{align}
for $p\geq 2$, $q>2$ such that
$$\frac{d}{p}+\frac{2}{q}<1.$$
The idea for the proof rests on a mixture of techniques in \cite{MMNPZ10} and \cite{FedFlan10}. More precisely, we approximate in the first step the drift coefficient $b$ by smooth functions $b_n$ with compact support and apply the It\^{o}-Tanaka-Zvonkin "trick" by transforming the solutions $X_{t}^{n,x}$ of \eqref{introSDE} associated with the coefficients $b_n$ to processes
$$Y_t^{n,x}:=X_t^{n,x} + U_n(t,X_t^{n,x}),$$
where the processes $Y_t^{n,x}$ satisfy an equation with more regular coefficients than \eqref{introSDE} given by
$$dY_t^{n,x} = \lambda U_n(t,X_t^{n,x})dt + \left( \mathcal{I}_d + \nabla U_n(t,X_ t^{n,x})\right) dB_t$$
for solutions $U_n$ to the backward PDE's
\begin{align}\label{backwardPDE}
\frac{\partial U_n}{\partial t} + \frac{1}{2} \Delta U_n + b_n \nabla U_n = \lambda U_n - b_n, \quad U_n(T,x)=0.
\end{align}

In the second step we use the compactness criterion for $L^2(\Omega)$ in \cite{DPMN92} applied to the sequence $Y_t^{n,x}$, $n\geq 1$ in connection with Schauder-type of estimates of solutions of \eqref{backwardPDE} and techniques from white noise analysis to show that
$$Y_t^{n,x} \xrightarrow{n \to \infty} Y_t^x$$
in $L^2(\Omega)$ for all $t$ and that
$$X_t^x = \varphi(t, Y_t^x),$$
where $\varphi(t,\cdot)$ is the inverse of the function $x\mapsto x+ U(t,x)$ for all $t$ and $U$ a solution of \eqref{backwardPDE}, is a Malliavin differentiable unique strong solution of \eqref{introSDE}.

Our paper is organised as follows:
In Section \ref{main results} we present our main results on the construction of strong solutions (Theorem \ref{thmainres1} and Theorem \ref{generalsde}). As an application of the results obtained in Section \ref{main results} we establish in Section \ref{applications} a Bismut-Elworthy-Li formula for the representation of first order derivatives of solutions of Kolmogorov equations.


\section{Main results}\label{main results}


In this section, we want to further develop the ideas introduced in \cite{FedFlan10} and \cite{MBP10} to derive Malliavin differentiable strong solutions of stochastic differential equations with irregular coefficients. More precisely, we aim at analyzing the SDE's of the form 

\begin{align}\label{eqmain1}
dX_t=b(t,X_t)dt + dB_t, \,\,0\leq t\leq 1,\,\,\, X_0=x  \in \mathbb{R}^d\,,
\end{align}
where the drift coefficient $b:[0,T]\times \mathbb{R}^{d}\longrightarrow 
\mathbb{R}^{d}$ is a Borel measurable function satisfying some integrability condition and $B_t$ is a $d$-dimensional Brownian motion with respect to the stochastic basis 
\begin{equation}
\left( \Omega ,\mathcal{F},\mu \right) ,\left\{ \mathcal{F}_{t}\right\}
_{0\leq t\leq T}  \label{Stochbasis}
\end{equation}%
for the $\mu -$augmented filtration $\left\{ \mathcal{F}_{t}\right\} _{0\leq
t\leq T}$ generated by $B_{t}$. At the end of this section we shall also apply our technique to equations with more general diffusion coefficients (Theorem \ref{generalsde}).

Consider the space
$$L_p^q:=L^q\left( [0,T], L^p(\R^d, \R^d) \right)$$
for $p,q\in \R$ satisfying the following condition
\begin{align}\label{pqcondition}
p> 2, \ q>2 \ \mbox{ and } \ \frac{d}{p} + \frac{2}{q} <1
\end{align}
and denote by $|\cdot|$ the Euclidean norm in $\R^d$. The Banach space $L_p^q$ is endowed with the norm
\begin{align}\label{intcondition}
\|f\|_{L_p^q}= \left( \int_0^T \left( \int_{\R^d} |f(t,x)|^p dx \right)^{q/p} dt \right)^{1/q}< \infty
\end{align}
for $f\in L_p^q$.

The main goal of the paper is to show that SDE's of the type (\ref{eqmain1}) with drift coefficient $b$ satisfying the integrability condition given in (\ref{intcondition}) admit strong solutions that are unique and in addition, Malliavin differentiable.

So, our main result is the following theorem:

\begin{thm} \label{thmainres1}
Suppose that the drift coefficient $b: [0,T] \times \mathbb{R}^d \rightarrow \mathbb{R}^d$ in \eqref{eqmain1} belongs to $L_p^q$. Then there exists a unique global strong solution $X$ to equation \eqref{eqmain1} such that $X_t$ is Malliavin differentiable for all $0 \leq t \leq T$.
\end{thm}

An important step of the proof of Theorem \ref{thmainres1} is directly based on the study of the regularity of solutions to the following associated PDE to equation SDE (\ref{eqmain1}).

\begin{align}\label{koleq}
\partial_t U(t,x) + b(t,x)\cdot \nabla U(t,x) + \frac{1}{2} \Delta U(t,x) - \lambda U(t,x) + b=0, \ \ t\in [0,T], \ \ U(T,x)=0,
\end{align}
where $U: [0,T]\times \R^d \rightarrow \R^d$, $\lambda>0$ and $b\in L_p^q$.

The following result is due to \cite{FedFlan12} and stablishes the well-posedness of the above PDE problem in a certain space.

First, recall the definition of the following functional spaces
$$\mathbb{H}_{\alpha,p}^q = L^q ([0,T], W^{\alpha,p}(\R^d)), \ \ \mathbb{H}_p^{\beta,q} = W^{\beta,q}([0,T], L^p(\R^d))$$
and
$$H_{\alpha,p}^{q} = \mathbb{H}_{\alpha,p}^q \cap \mathbb{H}_p^{1,q}.$$

The norm in $H_{\alpha,p}^{q}$ can be taken to be
$$\|u\|_{H_{\alpha,p}^{q}} \equiv \|u\|_{\mathbb{H}_{\alpha,p}^q}+ \|\partial_t u\|_{L_p^q}.$$

\begin{thm}\label{theoremU}
Let $p,q$ be such that $p\geq 2$, $q>2$ and $\frac{d}{p} + \frac{2}{q} <1$ and $\lambda>0$. Consider two vector fields $b,\Phi \in L_p^q$. Then there exists a unique solution of the backward parabolic system
\begin{align}\label{PDE(f)}
\partial_t u + \frac{1}{2} \Delta u + b\cdot \nabla u - \lambda u + \Phi = 0, \ \ t\in [0,T], \ \ u(T,x)=0
\end{align}
belonging to the space
$$H_{2,p}^q := L^q ([0,T], W^{2,p}(\R^d)) \cap W^{1,q} ([0,T], L^p(\R^d)),$$
i.e. there exists a constant $C>0$ depending only on $d,p,q,T,\lambda$ and $\|b\|_{L_p^q}$ such that
\begin{align}\label{ubound}
\|u\|_{H_{2,p}^q} \leq C \|\Phi\|_{L_p^q}.
\end{align}
\end{thm}

The following result is a part of \cite[Lemma 10.2]{KR05} that gives us some properties on the regularity of $u \in H_{2,p}^q$ that we will need for the proof of Theorem \ref{thmainres1}.

\begin{lemm}\label{Ureg}
Let $p,q\in (1,\infty)$ such that $\frac{d}{p}+\frac{2}{q}<1$ and $u \in H_{2,p}^q$, then $\nabla u$ is H\"{o}lder continuous in $(t,x)\in [0,T]\times \R^d$, namely for any $\varepsilon\in (0,1)$ satisfying
$$\varepsilon + \frac{d}{p}+\frac{2}{q}<1$$
there exists a constant $C>0$ depending only on $p,q$ and $\varepsilon$ such that for all $s,t\in [0,T]$ and $x,y\in \R^d$, $x\neq y$
\begin{align}\label{Uholder1}
\|\nabla u(t,x) - \nabla u(s,x)\| \leq C |t-s|^{\varepsilon/2}\|\nabla u\|_{H_{2,p}^q}^{1-1/q - \varepsilon/2}\|\partial_t u \|_{L_p^q}^{1/q+\varepsilon/2},
\end{align}
\begin{align}\label{Uholder2}
\|\nabla u(t,x)\|+\frac{\|\nabla u(t,x)- \nabla u(t,y)\|}{|x-y|^{\varepsilon}}\leq C T^{-1/q}\left(\|u\|_{H_{2,p}^q}+T \|\partial_t u\|_{L_p^q} \right),
\end{align}
where $\|\cdot\|$ denotes any norm in $\R^{d\times d}$
\end{lemm}

Our method to construct strong solutions is actually motivated by the following observation in \cite{LP04} and \cite{MBP06} (see also \cite{MBP09}).

\begin{prop}
\label{explicit} Suppose that the drift coefficient $b:[0,T]\times \mathbb{R}%
^{d}\mathbb{\longrightarrow R}^{d}$ in (\ref{eqmain1}) is bounded and
Lipschitz continuous. Then the unique strong solution $X_{t}=(X^1_{t},...,X^d_{t})$ of (\ref{eqmain1}) has the explicit representation 
\begin{equation}
\varphi \left( t,X_{t}^{i}(\omega )\right) =E_{\widetilde{\mu }}\left[
\varphi \left( t,\widetilde{B}_{t}^{i}(\widetilde{\omega })\right) 
\mathcal{E}_{T}^{\diamond }(b)\right]   \label{exprep}
\end{equation}%
for all $\varphi :[0,T]\times \mathbb{R\longrightarrow R}$ such that $%
\varphi \left( t,B_{t}^{i}\right) \in L^{2}(\Omega )$ for all $0\leq t\leq T,$
$i=1,\ldots,d,$. The random element $\mathcal{E}_T^{\diamond }(b)$ is given by%
\begin{align}
\mathcal{E}_{T}^{\diamond }(b)(\omega ,\widetilde{\omega })
:= &  \exp^{\diamond }\Big( \sum_{j=1}^{d}\int_{0}^{T}\left( W_{s}^{j}(\omega
)+b^{j}(s,\widetilde{B}_{s}(\widetilde{\omega }))\right) d\widetilde{B}%
_{s}^{j}(\widetilde{\omega })   \notag \\
&  -\frac{1}{2}\int_{0}^{T}\left( W_{s}^{j}(\omega )+b^{j}(s,%
\widetilde{B}_{s}(\widetilde{\omega }))\right) ^{\diamond 2}ds\Big) .
\label{Doleanswick}
\end{align}%
Here $\left( \widetilde{\Omega },\widetilde{\mathcal{F}},\widetilde{\mu }%
\right) ,\left( \widetilde{B}_{t}\right) _{t\geq 0}$ is a copy of the
quadruple $\left( \Omega ,\mathcal{F},\mu \right) ,$ $\left( B_{t}\right)
_{t\geq 0}$ in (\ref{Stochbasis})$.$ Further $E_{\widetilde{\mu }}$ denotes
a Pettis integral of random elements $\Phi :\widetilde{\Omega }$ $%
\longrightarrow \left( \mathcal{S}\right) ^{\ast }$ with respect to the
measure $\widetilde{\mu }.$ The Wick product $\diamond $ in the Wick
exponential of (\ref{Doleanswick}) (see \ref{Sprop}) is taken with respect to $\mu $ and $%
W_{t}^{j}$ is the white noise of $B_{t}^{j}$ in the Hida space $\left( 
\mathcal{S}\right) ^{\ast }$ (see (\ref{whitenoise}))$.$ The stochastic
integrals $\int_{0}^{T}\phi (t,\widetilde{\omega })d\widetilde{B}_{s}^{j}(%
\widetilde{\omega })$ in (\ref{Doleanswick}) are defined for predictable
integrands $\phi $ with values in the conuclear space $\left( \mathcal{S}%
\right) ^{\ast }$. See e.g. \cite{KX95} for definitions. The other
integral type in (\ref{Doleanswick}) is to be understood in the sense of
Pettis.
\end{prop}

\begin{rem}
\label{Remarkexp} Let $0=t_{1}^{n}<t_{2}^{n}<\ldots<t_{m_{n}}^{n}=T$\ be a
sequence of partitions of the interval $[0,T]$\ with $\max_{i=1}^{m_{n}-1}%
\left\vert t_{i+1}^{n}-t_{i}^{n}\right\vert \longrightarrow 0$\ . Then the
stochastic integral of the white noise $W^{j}$\ can be approximated as
follows:%
\begin{equation*}
\int_{0}^{T}W_{s}^{j}(\omega )d\widetilde{B}_{s}^{j}(\widetilde{\omega }%
)=\lim_{n\longrightarrow \infty }\sum_{i=1}^{m_{n}}(\widetilde{B}%
_{t_{i+1}^{n}}^{j}(\widetilde{\omega })-\widetilde{B}_{t_{i}^{n}}^{j}(%
\widetilde{\omega }))W_{t_{i}^{n}}^{j}(\omega )
\end{equation*}%
in $L^{2}(\lambda \times \widetilde{\mu };\left( \mathcal{S}\right) ^{\ast
}).$\ For more information about stochastic integration on conuclear spaces
the reader is referred to \cite{KX95}.
\end{rem}

In the sequel we shall use the notation $Y_{t}^{i,b}$ for the expectation on the right hand side of (\ref{exprep}) for $\varphi (t,x)=x$, that is%
\begin{equation*}
Y_{t}^{i,b}:=E_{\widetilde{\mu }}\left[ \widetilde{B}_{t}^{(i)}\mathcal{E}%
_{T}^{\diamond }(b)\right] 
\end{equation*}%
for $i=1,\ldots,d.$ We set
\begin{equation}
Y_{t}^{b}=\left( Y_{t}^{1,b},\ldots,Y_{t}^{d,b}\right)\,.   \label{equation}
\end{equation}%

The form of Formula \eqref{exprep} in Proposition \ref{explicit} actually gives rise to the conjecture that the expectation on the right hand side of $Y^b_t$ in \eqref{equation} may also define solutions of \eqref{eqmain1} for drift coefficients $b$ lying in $L_p^q$.
\ \

Our method to construct strong solutions to SDE (\ref{eqmain1}) which are Malliavin differentiable is essentially based on three steps.

\begin{itemize}
\item First, we consider a sequence of compactly supported smooth functions $b_n:[0,T]\times \R^d \rightarrow \R^d$, $n\geq 0$ such that $b_0:=b$ and $\sup_{n\geq 0} \|b_n\|_{L_p^q} <\infty$ approximating $b \in L_p^q$ a.e. with respect to the Lebesgue measure and then we prove that the sequence of strong solutions $X_t^n = Y_t^{b_n}$, $n \geq 1$, is relatively compact in $L^2(\Omega;\mathbb{R}^d)$ (Corollary \ref{relcompact}) for every $t\in [0,T]$. The main tool to verify compactness is the bound in Lemma \ref{estimate1} in connection with a compactness criterion in terms of Malliavin derivatives obtained in \cite{DPMN92} (see \mbox{Appendix \ref{App1}}). This step is one of the main contributions of this paper.

\item Secondly, given a merely measurable drift coefficient $b$ in the space $L_p^q$, we show that \mbox{$Y^b_t$}, $t\in [0,T]$ is a generalized process in the Hida distribution space and we invoke the $S$-transform (\ref{Stransform}) to prove that for a given sequence of a.e. approximating, smooth coefficients $b_n$ with compact support such that $\sup_{n\geq 0} \|b_n\|_{L_p^q}$, a subsequence of the corresponding strong solutions $X_t^{n_j} = Y_t^{b_{n_j}}$ fulfils 
$$
Y^{b_{n_j}}_t \rightarrow Y^b_t
$$
in $L^2(\Omega;\mathbb{R}^d)$ for $0 \leq t \leq T$ (Lemma \ref{squareint}).

\item Finally, using a certain transformation property for $Y^b_t$ (Lemma \ref{translemma}) we directly show that $Y^b_t$ is a Maaliavin differentiable solution to \eqref{eqmain1}.
\end{itemize}

We turn now to the first step of our procedure. The successful completion of the first step relies on the following essential lemma:

\begin{lemm}\label{estimate1}
Let $b_n:[0,T]\times \R^d \rightarrow \R^d$, $n\geq 1$ be a sequence of functions in $\mathcal{C}_0^{\infty}(\R^d)$ (space of infinitely often differentiable functions with compact support) approximating $b\in L_p^q$ a.e. such that $b_0:=b$ and $\sup_{n\geq 0} \|b_n\|_{L_p^q} < \infty$. Denote by $X_t^{n,x}$ the strong solution of SDE (\ref{eqmain1}) with drift coefficient $b_n$ for each $n\geq 0$. Then for every $t\in[0,T]$, $0\leq r' \leq r\leq t$ there exist a $0<\delta<1$ and a function $C:\R\rightarrow [0,\infty)$ depending only on $p,q,d,\delta$ and $T$ such that
\begin{align}\label{statement1}
E\left[\| D_{r'} X_t^{n,x} - D_r X_t^{n,x}\|^2\right] \leq  C(\|b_n\|_{L^{p,q}})|r'-r|^{\delta}
\end{align}
with
$$\sup_{n\geq 1} C(\|b_n\|_{L_p^q}) < \infty.$$
Here $\|\cdot \|$ denotes any norm in $\R^{d\times d}$.

Moreover,
\begin{align}\label{statement2}
\sup_{n\geq 1}\sup_{r\in [0,T]} E\left[\|D_r X_t^{n,x}\|^p\right] < \infty
\end{align}
for all $p\geq 2$.
\end{lemm}
\begin{proof}
Throughout the proof we will denote by $C_{\ast}:\R \rightarrow [0,\infty)$ any function depending on the parameters $\ast$. We will also use the symbol $\lesssim$ to denote \emph{less or equal} up to a positive real constant independent of $n$.

We will prove the above estimates by considering the solution of the associated PDE presented in (\ref{koleq}) with $b_n$, $n\geq 0$ in place of $b$ which we denote by $U_n$, $n\geq 0$ and then using the results introduced at the beginning of this section on the regularity of its solution.

First, let us introduce a new process that will be useful for this purpose. Consider for each $n\geq 0$ and $t\in [0,T]$ the functions $\gamma_{t,n} :\R^d \rightarrow \R^d$ defined as $\gamma_{t,n} (x) = x+U_n(t,x)$. It turns out, see \cite[Lemma 3.5]{FedFlan12}, that the functions $\gamma_{t,n}$, $t\in [0,T]$, $n\geq 0$ define a family of $C^1$-diffeomorphisms on $\R^d$. Furthermore, consider the auxiliary process $\tilde{X}_t^{n,x} := \gamma_{t,n}( X_t^{n,x})$, $t\in [0,T]$, $n\geq 1$. One checks using It\^{o}'s formula and (\ref{koleq}) that $\tilde{X}_t^{n,x}$ satisfies the following SDE
\begin{align}\label{SDEY}
d\tilde{X}_t^{n,x} = \lambda U_n(t, \gamma_{t,n}^{-1}( \tilde{X}_t^{n,x}))dt + \left( \mathcal{I}_d + \nabla U_n(t,\gamma_{t,n}^{-1} (\tilde{X}_t^{n,x})) \right) dB_t, \ \ \tilde{X}_0^{n,x} = x+U_n(0,x)
\end{align}
which is equivalent to SDE (\ref{eqmain1}) if we replace $b$ by $b_n$, $n\geq 1$. Using the chain rule for Malliavin derivatives (see e.g. \cite{Nua10}) we see that for $0\leq r\leq t$,
$$D_r \tilde{X}_t^{n,x} = \nabla\gamma_{t,n}( X_t^{n,x}) D_r X_t^{n,x}.$$
Because of Lemma \ref{lemma3} it suffices to prove the estimates (\ref{statement1}) and (\ref{statement2}) for the process $\tilde{X}_t^{n,x}$.

Since $b_n$ are now smooth we have that (\ref{SDEY}) admits a unique strong solution which takes the form
$$\tilde{X}_t^{n,x} =x+ U_n (0,x) + \lambda\int_0^t U_n(s, \gamma_{s,n}^{-1}(\tilde{X}_s^{n,x}))ds + \int_0^t \left( \mathcal{I}_d + \nabla U_n(s,\gamma_{s,n}^{-1} ( \tilde{X}_s^{n,x})) \right) dB_s. $$

Then the Malliavin derivative of $\tilde{X}_t^{n,x}$ for $0\leq r\leq t$, which exists (see e.g. \cite{Nua10}), is
\begin{align*}
D_r \tilde{X}_t^{n,x} =& \hspace{1mm} \mathcal{I}_d+ \nabla U_n(r,\gamma_{r,n}^{-1} ( \tilde{X}_r^{n,x}))\\
&+ \lambda\int_r^t \nabla U_n(s, \gamma_{s,n}^{-1}(\tilde{X}_s^{n,x}))\nabla \gamma_{s,n}^{-1}(\tilde{X}_s^{n,x}) D_r\tilde{X}_s^{n,x} ds\\
&+ \int_r^t \nabla^2 U_n(s, \gamma_{s,n}^{-1}(\tilde{X}_s^{n,x})) \nabla \gamma_{s,n}^{-1}(\tilde{X}_s^{n,x}) D_r\tilde{X}_s^{n,x}dB_s.
\end{align*}

Denote for simplicity, $Z_{r,t}^n := D_r \tilde{X}_t^{n,x}$. Then for $r'<r$ we can write
\begin{align*}
Z_{r',t}^n - Z_{r,t}^n &= \nabla U_n(r',\gamma_{r',n}^{-1}(\tilde{X}_{r'}^{n,x}))-\nabla U_n(r,\gamma_{r,n}^{-1}(\tilde{X}_{r}^{n,x}))\\
&+ \lambda \int_{r'}^r \nabla U_n(s, \gamma_{s,n}^{-1}(\tilde{X}_{s}^{n,x}))  \nabla \gamma_{s,n}^{-1}(\tilde{X}_{s}^{n,x}) Z_{r',s}^n ds\\
&+ \lambda \int_{r}^t \nabla U_n(s, \gamma_{s,n}^{-1}(\tilde{X}_{s}^{n,x})) \nabla \gamma_{s,n}^{-1}(\tilde{X}_{s}^{n,x}) \left(Z_{r',s}^n-Z_{r,s}^n\right)ds\\
&+ \int_{r'}^r \nabla^2 U_n(s, \gamma_{s,n}^{-1}(\tilde{X}_{s}^{n,x})) \nabla \gamma_{s,n}^{-1}(\tilde{X}_{s}^{n,x}) Z_{r',s}^n dB_s\\
&+ \int_{r}^t \nabla^2 U_n(s, \gamma_{s,n}^{-1}(\tilde{X}_{s}^{n,x})) \nabla \gamma_{s,n}^{-1}(\tilde{X}_{s}^{n,x}) \left(Z_{r',s}^n-Z_{r,s}^n\right)dB_s\\
&= Z_{r',r}^n - Z_{r,r}^n\\
&+ \lambda \int_{r}^t \nabla U_n(s, \gamma_{s,n}^{-1}(\tilde{X}_{s}^{n,x})) \nabla \gamma_{s,n}^{-1}(\tilde{X}_{s}^{n,x}) \left(Z_{r',s}^n-Z_{r,s}^n\right)ds\\
&+ \int_{r}^t \nabla^2 U_n(s, \gamma_{s,n}^{-1}(\tilde{X}_{s}^{n,x})) \nabla \gamma_{s,n}^{-1}(\tilde{X}_{s}^{n,x}) \left(Z_{r',s}^n-Z_{r,s}^n\right)dB_s.
\end{align*}

By dint of Lemma \ref{lemma2} we know that $\nabla U_n$ is bounded uniformly in $n$ and Lemma \ref{lemma1} shows that $\nabla^2 U_n$ belongs, at least, to $L_p^q$ uniformly in $n$. This implies that the stochastic integral in the expression for $Z_{r',t}^n - Z_{r,t}^n$ is a \emph{true} martingale, which we here denote my $M_t^n$. As a result, since the initial condition $Z_{r',r}^n - Z_{r,r}^n$  is $\mathcal{F}_r$-measurable for each $n\geq 0$, for a given $\alpha\geq 2$, by It\^{o}'s formula we have

\begin{align}\label{ItoZ1}
\begin{split}
\|Z_{r',t}^n - Z_{r,t}^n\|^\alpha &\lesssim  \|Z_{r',r}^n - Z_{r,r}^n \|^{\alpha} + \int_r^t  \|Z_{r',s}^n - Z_{r,s}^n \|^{\alpha}ds + M_t^n \\
&+ \int_r^t \|Z_{r',s}^n - Z_{r,s}^n \|^{\alpha-2}\tr \Bigg[ \left(\nabla^2 U_n(s, \gamma_{s,n}^{-1}(\tilde{X}_{s}^{n,x})) \nabla \gamma_{s,n}^{-1}(\tilde{X}_{s}^{n,x}) (Z_{r',s}^n - Z_{r,s}^n)\right)\\
&\times \left( \nabla^2 U_n(s, \gamma_{s,n}^{-1}(\tilde{X}_{s}^{n,x})) \nabla \gamma_{s,n}^{-1}(\tilde{X}_{s}^{n,x})(Z_{r',s}^n - Z_{r,s}^n) \right)^\ast \Bigg] ds
\end{split}
\end{align}
where here $\tr$ stands for the trace and $\ast$ for the transposition of matrices.

We proceed then using the fact that the trace of the matrix appearing in \eqref{ItoZ1} can be bounded by a constant $C_{p,d}$ independent of $n$, times $\|Z_{r',s}^n - Z_{r,s}^n\|^2 \|\nabla^2 U_n(s, \gamma_{s,n}^{-1}(\tilde{X}_{s}^{n,x})) \nabla \gamma_{s,n}^{-1}(\tilde{X}_{s}^{n,x})\|^2$.

Altogether,

\begin{align}\label{middleItoZ}
\begin{split}
\|Z_{r',t}^n - Z_{r,t}^n\|^\alpha &\lesssim \|Z_{r',r}^n - Z_{r,r}^n \|^{\alpha} + \int_r^t  \|Z_{r',s}^n - Z_{r,s}^n \|^{\alpha}ds + M_t^n \\
&+ \int_r^t \|Z_{r',s}^n - Z_{r,s}^n \|^{\alpha}\|\nabla^2 U_n(s, \gamma_{s,n}^{-1}(\tilde{X}_{s}^{n,x})) \nabla \gamma_{s,n}^{-1}(\tilde{X}_{s}^{n,x})\|^2 ds
\end{split}
\end{align}

Consider thus the process
\begin{align}\label{processV}
V_t^n := \int_r^t \|\nabla^2 U_n(s, \gamma_{s,n}^{-1}(\tilde{X}_{s}^{n,x})) \nabla \gamma_{s,n}^{-1}(\tilde{X}_{s}^{n,x})\|^2 ds.
\end{align}
The process $V_t^n$ is a continuous non-decreasing and $\{\mathcal{F}_t\}_{t\in[0,T]}$-adapted process such that $V_r^n =0$. Then Lemma \ref{lemma1} in connection with Theorem \ref{theoremU} we have that $\displaystyle \sup_{n\geq 0} E[V_t^n] <\infty$.

Then It\^{o}'s formula yields

\begin{align}\label{diff1}
e^{-V_t^n} \|Z_{r',t}^n - Z_{r,t}^n\|^\alpha  \lesssim \|Z_{r',r}^n - Z_{r,r}^n\|^\alpha + \int_r^t e^{-V_s^n} \|Z_{r',s}^n - Z_{r,s}^n\|^\alpha ds + \int_r^t e^{-V_s^n} dM_s.
\end{align}

Then taking expectation
\begin{align}\label{beforeGronwall}
E\left[ e^{-V_t^n} \| Z_{r',t}^n - Z_{r,t}^n \|^\alpha \right]   \lesssim  E\left[ \|Z_{r',r}^n - Z_{r,r}^n\|^\alpha\right]+ \int_r^t E\left[e^{-V_s^n} \|Z_{r',s}^n - Z_{r,s}^n\|^\alpha\right] ds.
\end{align}

Then Gronwall's inequality gives
\begin{align}\label{gronwall}
E\left[ e^{-V_t^n} \| Z_{r',t}^n - Z_{r,t}^n \|^\alpha \right]   &\lesssim  E\left[ \|Z_{r',r}^n - Z_{r,r}^n\|^\alpha\right].
\end{align}

At this point, it is easy to see, following similar steps, that for the process $Z_{r,t}^n$ one has
$$E\left[ e^{-V_t^n} \| Z_{r,t}^n \|^\alpha \right]   \lesssim  E\left[ \|Z_{r,r}^n\|^\alpha\right],$$
where $Z_{r,r}^n = \mathcal{I}_d + \nabla U_n (r,\gamma_{r,n}^{-1}(\tilde{X}_r^{n,x}))$. So
\begin{align}\label{Mallbound}
\sup_{n\geq 0} \sup_{r\in [0,T]} E\left[ e^{-V_t^n} \| Z_{r,t}^n \|^\alpha \right] \lesssim 1+\sup_{n\geq 0} \sup_{r\in [0,T]} E\left[ \|\nabla U_n (r,\gamma_{r,n}^{-1}(\tilde{X}_r^{n,x}))\|^\alpha\right] <\infty
\end{align}
because of Lemma \ref{lemma2} (ii) for a sufficiently large $\lambda\in \R$.

Then, the Cauchy-Schwarz inequality and Lemma \ref{lemma4} give
$$\sup_{n\geq 0} \sup_{r\in [0,T]} E\left[ \| Z_{r,t}^n \|^\alpha \right] \leq \sup_{n\geq 0} \sup_{r\in [0,T]} E\left[ e^{-2V_t^n}\| Z_{r,t}^n \|^{2\alpha} \right]^{1/2} \sup_{n\geq 0} E\left[ e^{2V_T^n}\right]^{1/2}  <\infty.$$

We continue to prove the estimate (\ref{statement1}). Recall that
\begin{align}\label{Zrr}
Z_{r',r}^n - Z_{r,r}^n &= \nabla U_n(r',\gamma_{r',n}^{-1}(\tilde{X}_{r'}^{n,x})) - \nabla U_n(r,\gamma_{r,n}^{-1}(\tilde{X}_{r}^{n,x})) \notag \\
&+ \lambda \int_{r'}^r \nabla U_n(s,\gamma_{s,n}^{-1}(\tilde{X}_{s}^{n,x}))\nabla\gamma_{s,n}^{-1}(\tilde{X}_{s}^{n,x}) Z_{r',s}ds  \\
&+ \int_{r'}^r \nabla^2 U_n(s,\gamma_{s,n}^{-1}(\tilde{X}_{s}^{n,x})) \nabla \gamma_{s,n}^{-1}(\tilde{X}_{s}^{n,x}) Z_{r',s}dB_s. \notag
\end{align}
Then taking norm and using Burkholder-Davis-Gundy inequality we get
\begin{align}\label{burkholder}
E \left[ \|Z_{r',r}^n - Z_{r,r}^n\|^\alpha \right] &\lesssim E\left[ \|\nabla U_n(r',\gamma_{r',n}^{-1}(\tilde{X}_{r'}^{n,x})) - \nabla U_n(r,\gamma_{r,n}^{-1}(\tilde{X}_{r}^{n,x}))\|^\alpha \right]\\
&+ \lambda^{\alpha} E\left[ \left(\int_{r'}^r \|\nabla U_n(s,\gamma_{s,n}^{-1}(\tilde{X}_{s}^{n,x}))\nabla\gamma_{s,n}^{-1}(\tilde{X}_{s}^{n,x}) Z_{r',s}\|ds\right)^\alpha\right] \notag \\
&+ E\left[ \left(\int_{r'}^r \|\nabla^2 U_n(s,\gamma_{s,n}^{-1}(\tilde{X}_{s}^{n,x})) \nabla \gamma_{s,n}^{-1}(\tilde{X}_{s}^{n,x}) Z_{r',s}\|^2 ds\right)^{\alpha/2} \right].\notag\\
&=: i)_n + ii)_n + iii)_n \notag
\end{align}

The aim now is to find H\"{o}lder bounds in the sense of \eqref{statement1} for the expressions appearing in (\ref{burkholder}).

For $i)_n$ we may write

\begin{align*}
i)_n &= E\left[ \|\nabla U_n(r',\gamma_{r',n}^{-1}(\tilde{X}_{r'}^{n,x})) - \nabla U_n(r,\gamma_{r,n}^{-1}(\tilde{X}_{r}^{n,x}))\|^\alpha \right]\\
&\lesssim E\left[ \|\nabla U_n(r',\gamma_{r',n}^{-1}(\tilde{X}_{r'}^{n,x})) - \nabla U_n(r,\gamma_{r',n}^{-1}(\tilde{X}_{r'}^{n,x}))\|^\alpha\right]\\
&+ E\left[\|\nabla U_n(r,\gamma_{r',n}^{-1}(\tilde{X}_{r'}^{n,x})) - \nabla U_n(r,\gamma_{r,n}^{-1}(\tilde{X}_{r}^{n,x}))\|^\alpha\right].
\end{align*}

Then by Lemma \ref{Ureg} there exists an $\varepsilon\in (0,1/ \alpha)$ and a constant $C_{p,q,d,\alpha}>0$ independent of $n\geq 0$ such that
\begin{align*}
E\bigg[ \|\nabla U_n(r',\gamma_{r',n}^{-1}(\tilde{X}_{r'}^{n,x})) &- \nabla U_n(r,\gamma_{r',n}^{-1}(\tilde{X}_{r'}^{n,x}))\|^\alpha\bigg]\\
&\leq C_{p,q,d,\alpha}\left(|r'-r|^{\varepsilon/2} \|\nabla U_n\|_{H_{2,p}^q}^{1-1/q-\varepsilon/2} \|\partial_t U_n\|_{L_p^q}^{1/q+\varepsilon/2}\right)^{\alpha}
\end{align*}
and
\begin{align*}
E\bigg[\|\nabla U_n(r,\gamma_{r',n}^{-1}&(\tilde{X}_{r'}^{n,x})) - \nabla U_n(r,\gamma_{r,n}^{-1}(\tilde{X}_{r}^{n,x}))\|^\alpha\bigg]\\
&\leq C_{p,q,d,\alpha} T^{-\alpha/q} E\left[|\gamma_{r',n}^{-1}(\tilde{X}_{r'}^{n,x}) -\gamma_{r,n}^{-1}(\tilde{X}_{r}^{n,x})|^{\alpha \varepsilon}\right] \left( \| U_n\|_{H_{2,p}^q} + T \|\partial_t U_n\|_{L_p^q}\right)^{\alpha}.
\end{align*}

The above bounds in connection with inequality (\ref{ubound}) in Theorem \ref{theoremU} give
$$i)_n \leq C_{p,q,d,\alpha,T}(\|b_n\|_{L_p^q}) \left(|r'-r|^{\alpha\varepsilon/2} + E\left[|\gamma_{r',n}^{-1}(\tilde{X}_{r'}^{n,x}) -\gamma_{r,n}^{-1}(\tilde{X}_{r}^{n,x})|^{\alpha \varepsilon}\right]\right)$$
for some continuous function $C_{p,q,d,\alpha,T}(\cdot)$ and hence
$$\sup_{n\geq 0} C_{p,q,d,\alpha,T}(\|b_n\|_{L_p^q}) <\infty.$$

Moreover, using Girsanov's theorem, we obtain that
\begin{align*}
E\bigg[|\gamma_{r',n}^{-1}(\tilde{X}_{r'}^{n,x}) &- \gamma_{r,n}^{-1}(\tilde{X}_{r}^{n,x})|\bigg]=E\left[ |X_{r'}^{n,x}- X_{r}^{n,x} |\right]\\
&\lesssim  E\left[ \left|\int_{r'}^r b_n(s, x+B_s)ds\right| \mathcal{E}\left(\int_0^T b_n(u,x+B_u)dB_u\right)\right]+ E\left[ |B_{r'}-B_r|\right]\\
&\lesssim |r'-r|^{1/2} E\left[\int_{r'}^r |b_n(s,x+B_s)|^2 ds\right]^{1/2}+|r'-r|^{1/2}\\
&\lesssim |r'-r|^{1/2}
\end{align*},
where we used, Cauchy-Schwarz inequality and both that
$$\sup_{n\geq 0} E\left[\mathcal{E}\left(\int_0^T b_n(u,x+B_u)dB_u\right)^2\right]<\infty$$
and
$$\sup_{n\geq 0} E\left[\int_{r'}^r |b_n(s,x+B_s)|^2 ds\right]^{1/2}<\infty,$$
see \cite[Lemma 3.2]{KR05} or Lemma \ref{expmoments}.

By Jensen's inequality for concave functions and the previous estimate we have

$$E\left[|\gamma_{r',n}^{-1}(\tilde{X}_{r'}^{n,x}) -\gamma_{r,n}^{-1}(\tilde{X}_{r}^{n,x})|^{\alpha \varepsilon}\right] \leq E\left[|\gamma_{r',n}^{-1}(\tilde{X}_{r'}^{n,x}) -\gamma_{r,n}^{-1}(\tilde{X}_{r}^{n,x})|\right]^{\alpha \varepsilon}\lesssim |r'-r|^{\alpha \varepsilon /2}.$$

Altogether,
$$i)_n \leq C_{p,q,d,\alpha,T}(\|b_n\|_{L_p^q}) |r'-r|^{\delta}$$
for a $\delta \in (0,1)$.

For the second term, $ii)_n$, we use H\"{o}lder's inequality, Lemma \ref{lemma2} (ii) for a sufficiently large $\lambda\in \R$, Lemma \ref{lemma3} and the estimate (\ref{statement2}) to obtain

\begin{align*}
ii)_n &\lesssim \lambda^{\alpha} |r'-r|^{\alpha-1} \left( \sup_{s\in [0,t]}E\left[\|\nabla U_n(s,\gamma_{s,n}^{-1}(\tilde{X}_{s}^{n,x})) \nabla \gamma_{s,n}^{-1}(\tilde{X}_{s}^{n,x}) \|^{2\alpha}\right]ds \right)^{1/2} \left(\int_{r'}^r E\left[ \|Z_{r',s}^n\|^{2\alpha}\right]ds\right)^{1/2}\\
&\leq C_{p,q,d,\alpha,T} |r'-r|^{\delta}
\end{align*}
for a $\delta \in (0,1)$.

Finally, for the third term, for $\alpha\geq 2$, we use H\"{o}lder's inequality to obtain

\begin{align*}
iii)_n \lesssim    |r'-r|^{\frac{\alpha-2}{2}} E\left[\int_{r'}^r \|\nabla^2 U_n(s,\gamma_{s,n}^{-1}(\tilde{X}_{s}^n))\|^\alpha \|\nabla \gamma_{s,n}^{-1}(\tilde{X}_{s}^n))\|^{\alpha} \|Z_{r',s}^n \|^{\alpha} ds \right].
\end{align*}
Then choose $\alpha = 2(1+\delta)$ with $\delta \in (0,1/4)$ and use Lemma \ref{lemma3} to get

$$iii)_n \lesssim |r'-r|^{\delta} E\left[\int_{r'}^r \|\nabla^2 U_n(s,\gamma_{s,n}^{-1}(\tilde{X}_{s}^n))\|^{2(1+\delta)} \|Z_{r',s}^n \|^{2(1+\delta)} ds \right].$$
Then Fubini's theorem, H\"{o}lder's inequality once more with respect to $\mu(d\omega)$, with exponent $1+\delta'$, \mbox{$\delta' \in (0,1/4)$} and Cauchy-Schwarz yield
\begin{align*}
E\bigg[\int_{r'}^r & \|\nabla^2 U_n(s,\gamma_{s,n}^{-1}(\tilde{X}_{s}^n))\|^{2(1+\delta)} \|Z_{r',s}^n \|^{2(1+\delta)} ds \bigg]=\int_{r'}^r E\left[\|\nabla^2 U_n(s,\gamma_{s,n}^{-1}(\tilde{X}_{s}^n))\|^{2(1+\delta)} \|Z_{r',s}^n \|^{2(1+\delta)}\right] ds\\
&\lesssim \int_{r'}^r E\left[\|\nabla^2 U_n(s,\gamma_{s,n}^{-1}(\tilde{X}_{s}^n))\|^{2(1+\delta)(1+\delta')}\right]^{1/(1+\delta')} E\left[\|Z_{r',s}^n \|^{2(1+\delta)\frac{1+\delta'}{\delta'}}\right]^{\frac{\delta'}{1+\delta'}} ds\\
&\lesssim \sup_{n\geq 0}\sup_{s\in [r',r]} E\left[\|Z_{r',s}^n \|^{2(1+\delta)\frac{1+\delta'}{\delta'}}\right]^{\frac{\delta'}{1+\delta'}}\int_{r'}^r E\left[\|\nabla^2 U_n(s,\gamma_{s,n}^{-1}(\tilde{X}_{s}^n))\|^{2(1+\delta)(1+\delta')}\right]^{1/(1+\delta')} ds\\
&\lesssim \int_{0}^T E\left[\|\nabla^2 U_n(s,\gamma_{s,n}^{-1}(\tilde{X}_{s}^n))\|^{2(1+\delta)(1+\delta')}\right]^{1/(1+\delta')} ds
\end{align*}
where the last step follows from (\ref{statement2}). For the last factor, since $0<1/(1+\delta')<1$, using the inverse Jensen's inequality and the fact that $1<(1+\delta)(1+\delta')<2$ for suitable $\delta, \delta' \in (0,1/4)$ in connection with Lemma \ref{lemma1} we have
\begin{align*}
\int_{0}^T E\bigg[&\|\nabla^2 U_n(s,\gamma_{s,n}^{-1}(\tilde{X}_{s}^n))\|^{2(1+\delta)(1+\delta')}\bigg]^{1/(1+\delta')} ds\\
&\leq T^{1-1/(1+\delta')} \left(E\left[\int_{0}^T \|\nabla^2 U_n(s,\gamma_{s,n}^{-1}(\tilde{X}_{s}^n))\|^{2(1+\delta)(1+\delta')} ds\right]\right)^{1/(1+\delta')}\leq M<\infty
\end{align*}
for every $n\geq 0$, w.r.t. a constant $M$.

As a summary, it follows from \eqref{gronwall} that

\begin{align*}\label{diff1}
E\left[ e^{-V_t^n} \| Z_{r',t}^n - Z_{r,t}^n \|^{2(1+\delta)} \right] \leq C_{p,q,d,\alpha,T}(\|b_n\|_{L_p^q}) |r'-r|^{\delta}.
\end{align*}

Then by H\"{o}lder's inequality with exponent $1+\delta$, $\delta\in (0,1)$ together with Lemma \ref{lemma4} we obtain
\begin{align*}
E\left[\| Z_{r',t}^n - Z_{r,t}^n \|^2 \right] &= E\left[ e^{\frac{1}{1+\delta}V_t^n} e^{-\frac{1}{1+\delta}V_t^n} \| Z_{r',t}^n - Z_{r,t}^n \|^2 \right]\\
&\leq E\left[e^{\frac{1}{\delta}V_t^n} \right]^{\frac{\delta}{1+\delta}}E\left[ e^{-V_t^n} \| Z_{r',t}^n - Z_{r,t}^n \|^{2(1+\delta)}\right]^{\frac{1}{1+\delta}}\\
&\leq C_{p,q,d,\alpha,T}(\|b_n\|_{L_p^q}) |r'-r|^{\delta/(1+\delta)}
\end{align*}
with
$$\sup_{n\geq 0} C_{p,q,d,\alpha,T}(\|b_n\|_{L_p^q}) < \infty.$$
\end{proof}

\begin{rem}
The bound given in (\ref{statement2}) is in fact uniform in $x\in \R^d$. Indeed, by Lemma \ref{lemma2} item (ii) we have that the bound given in (\ref{Mallbound}) is also uniform in $x\in \R^d$. Moreover, since $\Delta U_n\in L_p^q$ for all $n\geq 0$, then by Lemma \ref{lemma2} item (iii) in connection with Lemma \ref{expmoments} we have that for any $k\in \R$
$$\sup_{x\in \R^d} \sup_{n\geq 0} E[e^{kV_T^n}] < \infty.$$
Hence, for any $\alpha\geq 1$
$$\sup_{x\in \R^d} \sup_{r\in [0,T]}\sup_{n\geq 0} E\left[\|D_r X_t^{n,x}\|^{\alpha}\right]<\infty.$$
\end{rem}

\begin{rem}\label{derivativebound}
One also checks that the same holds for the spatial derivatives, that is for any $\alpha\geq 1$
$$\sup_{x\in \R^d} \sup_{r\in [0,T]}\sup_{n\geq 0} E\left[\|\frac{\partial}{\partial x} X_t^{n,x}\|^{\alpha}\right]<\infty$$
by using the fact that $\frac{\partial}{\partial x} X_t^{n,x}$ solves the same SDE as $D_r X_t^{n,x}$, starting at $r=0$.
\end{rem}

As a repercussion of Lemma \ref{estimate1} we have the following result which is central in the proof of the existence of strong solutions of (\ref{eqmain1}).

\begin{cor}\label{relcompact}
Let $\{b_n\}_{n\geq 0}$ be a sequence of compactly supported smooth functions approximating $b$ in $L_p^q$. Denote, as before, $X_t^{x,n}$ the solution to equation (\ref{eqmain1}) with drift coefficient $b_n$. Then for each $t\in [0,T]$ the sequence of random variables $X_t^{n,x}$, $n\geq 0$ is relatively compact in $L^2(\Omega)$.
\end{cor}
\begin{proof}
This is a direct consequence of the compactness criterion that can be found in \mbox{Appendix \ref{App2}}, Lemma \ref{MCompactness} and \ref{DaPMN}, which is due to \cite{DPMN92}, together with Lemma \ref{estimate1}. One can check that the double integral in Lemma \ref{DaPMN} is finite. Namely
$$ \int_0^T \int_0^T \frac{E\left[ \|Z_{r',t}^n - Z_{r,t}\|^2\right]}{|r'-r|^{1+2\beta}}dr'dr \leq \int_0^T\int_0^T \frac{1}{|r'-r|^{2\beta +1-\delta}} dr' dr < \infty$$
for any $0<\delta< 1$ and $2\beta+1-\delta<1$.
\end{proof}

The following lemma gives a criterion under which the process $Y_{t}^{b}$
belongs to the Hida distribution space.

\begin{lemm}
\label{hidadistr} Suppose that 
\begin{equation} 
E_{\mu }\left[ \exp \left( 36\int_{0}^{T}\left| b(s,B_{s})\right|^{2}ds\right) \right] <\infty ,  \label{intcond}
\end{equation}%
where the drift $b:[0,1]\times \mathbb{R}^{d}\mathbb{\longrightarrow R}^{d}$
is measurable (in particular, \eqref{intcond} is valid for $b\in L_p^q$ because of Lemma (\ref{expmoments})). Then the coordinates of the process $Y_{t}^{b},$ defined in (\ref{equation}), that is 
\begin{equation}
Y_{t}^{i,b}=E_{\widetilde{\mu }}\left[ \widetilde{B}_{t}^{(i)}\mathcal{E}%
_{T}^{\diamond }(b)\right]  \,,  \label{explobject} 
\end{equation}%
\bigskip are elements of the Hida distribution space.
\end{lemm}

\begin{proof}
See \cite{MBP10} for a similar proof.
\end{proof}

\begin{lemm}
\label{diffestimate} Let $\varepsilon \in (0,1)$ and define $p_{\varepsilon} := 1+\varepsilon$ and $q_{\varepsilon} := \frac{1+\varepsilon}{\varepsilon}$. Let $b_{n}:[0,T]\times \mathbb{R}^{d}\mathbb{%
\longrightarrow R}^{d}$ be a sequence of Borel measurable functions with $%
b_{0}=b$ such that
\begin{equation} \label{supbound}
\sup_{n \geq 0} E \left[ \exp \left( 16 q_{\varepsilon}(8 q_{\varepsilon}-1)\int_0^T \left| b_n(s,B_s) \right|^2 ds \right) \right] < \infty
\end{equation}
holds. Then 
\begin{equation*}
\left\vert S(Y_{t}^{i,b_{n}}-Y_{t}^{i,b})(\phi )\right\vert \leq const\cdot
E[J_{n}]^{\frac{1}{p_{\varepsilon}}}\cdot \exp \left(2(8 q_{\varepsilon}-1)\int_{0}^{T}\left| \phi (s)\right|^{2}ds\right)
\end{equation*}%
for all $\phi \in (S_{\mathbb{C}}([0,1]))^{d}$, $i=1,\ldots,d$, where $S$ denotes the $S$-transform (see Section \ref{WN} in Appendix \ref{mallcalgaus}) and where the factor 
$J_{n}$ is defined by 
\begin{align} \label{Jn}
J_n = \sum_{j=1}^d 2\left|\int_0^T (b_n^{(j)}(s,\tilde{B}_s^{(j)})-b^{(j)}(s,\tilde{B}_s^{(j)}))^2ds\right|^{\frac{p_{\varepsilon}}{2}}+\left|\int_0^T (b^{(j)}(s,\tilde{B}_s^{(j)})^2 - b_n^{(j)}(s,\tilde{B}_s^{(j)})^2)ds\right|^{p_{\varepsilon}}.
\end{align}
Here $S_{\mathbb{C}}([0,1])$ is the complexification of the Schwarz space $S([0,1])$ on $[0,1]$, see Section \ref{WN} in Appendix \ref{mallcalgaus}.

In particular, if $b_n$ approximates $b$ in the following sense
\begin{equation} \label{Jnestimate}
E[J_n] \rightarrow 0
\end{equation}
as $n \rightarrow \infty$, it follows that 
$$
Y^{b_n}_t \rightarrow Y^b_t \textrm{ in } ( \mathcal{S})^*
$$
as $n \rightarrow \infty$ for all $0 \leq t \leq 1$, $i = 1,\dots, d$.
\end{lemm}
\begin{proof}
For $i=1,\dots, d$ we obtain by Proposition \ref{explicit} and (\ref{StransfvonW}) that
\begin{align*}
|S(Y_t^{i,b_n}-Y_t^{i,b})(\phi)| &\leq E_{\tilde{\mu}}\Bigg[|\tilde{B}_t^{(i)}| \exp \Bigg\{\sum_{j=1}^d \mbox{Re}\bigg[\int_0^T (b^{(j)}(s,\tilde{B}_s^{(j)})+\phi^{(j)}(s))d\tilde{B}_s^{(j)}\\
&- \frac{1}{2}\int_0^T (b^{(j)}(s,\tilde{B}_s^{(j)})+\phi^{(j)}(s))^2 ds\bigg] \Bigg\}\\
&\times \Bigg|\exp \Bigg\{\sum_{j=1}^d \int_0^T (b_n^{(j)}(s,\tilde{B}_s^{(j)})-b^{(j)}(s,\tilde{B}_s^{(j)}))d\tilde{B}_s^{(j)} \\
&+\frac{1}{2}\int_0^T (b^{(j)}(s,\tilde{B}_s^{(j)})^2 - b_n^{(j)}(s,\tilde{B}_s^{(j)})^2)ds\\
&+\int_0^T \phi^{(j)}(s)(b^{(j)}(s,\tilde{B}_s^{(j)})-b_n^{(j)}(s,\tilde{B}_s^{(j)}))ds\Bigg\} - 1 \Bigg|\Bigg].
\end{align*}

Since $|\exp \{z\}-1| \leq |z| \exp\{|z|\}$ it follows from H\"{o}lder's inequality with exponents $p_{\varepsilon}=1+\varepsilon$ and $q_{\varepsilon} = \frac{1+\varepsilon}{\varepsilon}$, for an appropriate $\varepsilon>0$, that

\begin{align*}
|S(Y_t^{i,b_n}-Y_t^{i,b})(\phi)| &\leq E_{\tilde{\mu}} \left[|Q_n|^{p_{\varepsilon}}\right]^{\frac{1}{p_{\varepsilon}}} E_{\tilde{\mu}} \Bigg[ \bigg( |\tilde{B}_t^{(i)}| \exp \Bigg\{\sum_{j=1}^d \mbox{Re}\bigg[\int_0^T (b^{(j)}(s,\tilde{B}_s^{(j)})+\phi^{(j)}(s))d\tilde{B}_s^{(j)}\\
&- \frac{1}{2}\int_0^T (b^{(j)}(s,\tilde{B}_s^{(j)})+\phi^{(j)}(s))^2 ds\bigg] \Bigg\} \bigg)^{q_{\varepsilon}} \exp\left\{q_{\varepsilon} |Q_n|\right\}\Bigg]^{\frac{1}{q_{\varepsilon}}},
\end{align*}
where
\begin{align*}
Q_n &= \sum_{j=1}^d \int_0^T (b_n^{(j)}(s,\tilde{B}_s^{(j)})-b^{(j)}(s,\tilde{B}_s^{(j)}))d\tilde{B}_s^{(j)}+\frac{1}{2}\int_0^T (b^{(j)}(s,\tilde{B}_s^{(j)})^2 - b_n^{(j)}(s,\tilde{B}_s^{(j)})^2)ds\\
&+\int_0^T \phi^{(j)}(s)(b^{(j)}(s,\tilde{B}_s^{(j)})-b_n^{(j)}(s,\tilde{B}_s^{(j)}))ds.
\end{align*}

Then using the Cauchy-Schwarz inequality on the last integral and the fact that $|x|\leq e^x$ and $1\leq e^x$ for $x\geq 0$ we may write
\begin{align*}
E_{\tilde{\mu}}\left[ |Q_n|^{p_{\varepsilon}}\right] &\leq C \exp \left\{\left(\int_0^T |\phi(s)|^2 ds\right)^{p_{\varepsilon}/2}\right\} E_{\tilde{\mu}} \Bigg[ \sum_{j=1}^d \left|\int_0^T (b_n^{(j)}(s,\tilde{B}_s^{(j)})-b^{(j)}(s,\tilde{B}_s^{(j)}))d\tilde{B}_s^{(j)}\right|^{p_{\varepsilon}}\\
&+\left|\frac{1}{2}\int_0^T (b^{(j)}(s,\tilde{B}_s^{(j)})^2 - b_n^{(j)}(s,\tilde{B}_s^{(j)})^2)ds\right|^{p_{\varepsilon}}\\
&+\left|\int_0^T (b^{(j)}(s,\tilde{B}_s^{(j)})-b_n^{(j)}(s,\tilde{B}_s^{(j)}))^2ds\right|^{\frac{p_{\varepsilon}}{2}}\Bigg]\\
&= C \exp \left\{\left(\int_0^T |\phi(s)|^2 ds\right)^{p_{\varepsilon}/2}\right\} E_{\tilde{\mu}} \Bigg[ \sum_{j=1}^d 2\left|\int_0^T (b_n^{(j)}(s,\tilde{B}_s^{(j)})-b^{(j)}(s,\tilde{B}_s^{(j)}))^2ds\right|^{\frac{p_{\varepsilon}}{2}}\\
&+\left|\int_0^T (b^{(j)}(s,\tilde{B}_s^{(j)})^2 - b_n^{(j)}(s,\tilde{B}_s^{(j)})^2)ds\right|^{p_{\varepsilon}}\Bigg],
\end{align*}
where in the last inequality we used the Burkholder-Davis-Gundy inequality for the stochastic integral. Then
$$E_{\tilde{\mu}}\left[ |Q_n|^{p_{\varepsilon}}\right]^{\frac{1}{p_{\varepsilon}}} \leq  C \exp \left\{\frac{1}{p_{\varepsilon}}\left(\int_0^T |\phi(s)|^2 ds\right)^{p_{\varepsilon}/2}\right\} E_{\tilde{\mu}}\left[ J_n\right]^{\frac{1}{p_{\varepsilon}}},$$
where
$$J_n = \sum_{j=1}^d 2\left|\int_0^T (b_n^{(j)}(s,\tilde{B}_s^{(j)})-b^{(j)}(s,\tilde{B}_s^{(j)}))^2ds\right|^{\frac{p_{\varepsilon}}{2}}+\left|\int_0^T (b^{(j)}(s,\tilde{B}_s^{(j)})^2 - b_n^{(j)}(s,\tilde{B}_s^{(j)})^2)ds\right|^{p_{\varepsilon}}.$$

Further we get that
\begin{align*}
E_{\tilde{\mu}} \Bigg[ &\bigg( |\tilde{B}_t^{(i)}| \exp \Bigg\{\sum_{j=1}^d \mbox{Re}\bigg[\int_0^T (b^{(j)}(s,\tilde{B}_s^{(j)})+\phi^{(j)}(s))d\tilde{B}_s^{(j)}\\
&- \frac{1}{2}\int_0^T (b^{(j)}(s,\tilde{B}_s^{(j)})+\phi^{(j)}(s))^2 ds\bigg] \Bigg\} \bigg)^{q_{\varepsilon}} \exp\left\{q_{\varepsilon} |Q_n|\right\}\Bigg]^{\frac{1}{q_{\varepsilon}}}\\
&\leq E_{\tilde{\mu}} \Bigg[\bigg( |\tilde{B}_t^{(i)}| \exp \Bigg\{\sum_{j=1}^d \mbox{Re}\bigg[\int_0^T (b^{(j)}(s,\tilde{B}_s^{(j)})+\phi^{(j)}(s))d\tilde{B}_s^{(j)}\\
&- \frac{1}{2}\int_0^T (b^{(j)}(s,\tilde{B}_s^{(j)})+\phi^{(j)}(s))^2 ds\bigg] \Bigg\} \bigg)^{2 q_{\varepsilon}}\Bigg]^{\frac{1}{2 q_{\varepsilon}}} E_{\tilde{\mu}} \Bigg[ \exp\left\{2 q_{\varepsilon} |Q_n|\right\}\Bigg]^{\frac{1}{2 q_{\varepsilon}}}.
\end{align*}

Then for $z\in \mathbb{C}$ one has $\exp\{|z|\}\leq \frac{1}{2}\left(\exp\{2\mbox{Re } z\}+\exp\{-2\mbox{Re }z\}+\exp\{2\mbox{Im }z\}+\exp\{-2\mbox{Im }z\}\right)$. Thus
\begin{align*}
E_{\tilde{\mu}} \Bigg[ \exp\left\{2 q_{\varepsilon} |Q_n|\right\}\Bigg]^{\frac{1}{2 q_{\varepsilon}}} &\leq \frac{1}{2^{2q_{\varepsilon}}} \Bigg(E_{\tilde{\mu}} \Bigg[ \exp\left\{4 q_{\varepsilon} \mbox{Re } Q_n\right\}\Bigg]^{\frac{1}{2 q_{\varepsilon}}} + E_{\tilde{\mu}} \Bigg[ \exp\left\{-4 q_{\varepsilon} \mbox{Re } Q_n\right\}\Bigg]^{\frac{1}{2 q_{\varepsilon}}}\\
&+ E_{\tilde{\mu}} \Bigg[ \exp\left\{4 q_{\varepsilon} \mbox{Im } Q_n\right\}\Bigg]^{\frac{1}{2 q_{\varepsilon}}} + E_{\tilde{\mu}} \Bigg[ \exp\left\{-4 q_{\varepsilon} \mbox{Im } Q_n\right\}\Bigg]^{\frac{1}{2 q_{\varepsilon}}}\Bigg).
\end{align*}

By the Cauchy-Schwarz inequality and the supermartingale property of Dol\'{e}ans-Dade exponentials we get
\begin{align*}
E_{\tilde{\mu}} \Bigg[ \exp\left\{4 q_{\varepsilon} \mbox{Re } Q_n\right\}\Bigg]& \leq E_{\tilde{\mu}} \Bigg[\exp \Bigg\{\sum_{j=1}^d 32 q_{\varepsilon}^2 \int_0^T (b_n^{(j)}(s,\tilde{B}_s^{(j)})-b^{(j)}(s,\tilde{B}_s^{(j)}))^2ds\\
&+4 q_{\varepsilon}\int_0^T (b^{(j)}(s,\tilde{B}_s^{(j)})^2-b_n^{(j)}(s,\tilde{B}_s^{(j)})^2) ds\\
&+ 8q_{\varepsilon} \int_0^T \mbox{Re }\phi^{(j)}(s) (b^{(j)}(s,\tilde{B}_s^{(j)})-b_n^{(j)}(s,\tilde{B}_s^{(j)})) ds\bigg] \Bigg\}\Bigg]^{\frac{1}{2}}\\
&\leq L_n \exp \left\{2q_{\varepsilon} \int_0^T |\phi(s)|^2 ds \right\},
\end{align*}
where the last step follows from the fact that $\langle f,g \rangle \leq \frac{1}{2}(\|f\|^2 + \|g\|^2)$, $f,g\in L^2([0,T])$ and where
\begin{align*}
L_n &= E_{\tilde{\mu}} \Bigg[\exp \Bigg\{\sum_{j=1}^d 4q_{\varepsilon}(8 q_{\varepsilon}+1) \int_0^T (b_n^{(j)}(s,\tilde{B}_s^{(j)})-b^{(j)}(s,\tilde{B}_s^{(j)}))^2ds\\
&+4 q_{\varepsilon}\int_0^T (b^{(j)}(s,\tilde{B}_s^{(j)})^2-b_n^{(j)}(s,\tilde{B}_s^{(j)})^2) ds\bigg] \Bigg\}\Bigg]^{\frac{1}{2}}.
\end{align*}

Similarly, one also obtains
$$E_{\tilde{\mu}} \Bigg[ \exp\left\{-4 q_{\varepsilon} \mbox{Re } Q_n\right\}\Bigg] \leq L_n \exp \left\{2q_{\varepsilon} \int_0^T |\phi(s)|^2 ds \right\}.$$

In the same way, one also obtains the same bounds for $E_{\tilde{\mu}}\left[\exp\{4q_{\varepsilon} \mbox{Im } Q_n\} \right]$ and $E_{\tilde{\mu}}\left[\exp\{-4q_{\varepsilon} \mbox{Im } Q_n\} \right]$.

Finally, for the remaining factor we see that
\begin{align*}
E_{\tilde{\mu}} &\Bigg[\bigg( |\tilde{B}_t^{(i)}| \exp \Bigg\{\sum_{j=1}^d \mbox{Re}\bigg[\int_0^T (b^{(j)}(s,\tilde{B}_s^{(j)})+\phi^{(j)}(s))d\tilde{B}_s^{(j)}\\
&- \frac{1}{2}\int_0^T (b^{(j)}(s,\tilde{B}_s^{(j)})+\phi^{(j)}(s))^2 ds\bigg] \Bigg\} \bigg)^{2 q_{\varepsilon}}\Bigg]^{\frac{1}{2 q_{\varepsilon}}}\\
&\leq E_{\tilde{\mu}} \left[ |\tilde{B}_t^{(i)}|^{4q_{\varepsilon}}\right]^{\frac{1}{4q_{\varepsilon}}} E_{\tilde{\mu}} \Bigg[\exp \Bigg\{4 q_{\varepsilon}\sum_{j=1}^d  \mbox{Re}\bigg[\int_0^T (b^{(j)}(s,\tilde{B}_s^{(j)})+\phi^{(j)}(s))d\tilde{B}_s^{(j)}\\
&- \frac{1}{2}\int_0^T (b^{(j)}(s,\tilde{B}_s^{(j)})+\phi^{(j)}(s))^2 ds\bigg] \Bigg\}\Bigg]^{\frac{1}{4 q_{\varepsilon}}}\\
&\leq E_{\tilde{\mu}} \left[ |\tilde{B}_t^{(i)}|^{4q_{\varepsilon}}\right]^{\frac{1}{4q_{\varepsilon}}}E_{\tilde{\mu}} \Bigg[\exp \Bigg\{\sum_{j=1}^d   4 q_{\varepsilon}(8 q_{\varepsilon}-1)\int_0^T \mbox{Re }(b^{(j)}(s,\tilde{B}_s^{(j)})+\phi^{(j)}(s))^2ds \Bigg\}\Bigg]^{\frac{1}{4 q_{\varepsilon}}}.
\end{align*}

Now, since $\mbox{Re }(z^2) \leq (\mbox{Re } z)^2$, $z\in \mathbb{C}$ we have that $\mbox{Re }(b+\phi)^2 \leq (b+ \mbox{Re } \phi)^2$ then using Minkowski's inequality, i.e. $\|f+g\|_p^p \leq 2^{p-1}(\|f\|_p^p + \|g\|_p^p)$ for any $p\geq 1$ and Cauchy-Schwarz inequality w.r.t. $\tilde{\mu}$ one finally obtains

\begin{align*}
E_{\tilde{\mu}} &\Bigg[\bigg( |\tilde{B}_t^{(i)}| \exp \Bigg\{\sum_{j=1}^d \mbox{Re}\bigg[\int_0^T (b^{(j)}(s,\tilde{B}_s^{(j)})+\phi^{(j)}(s))d\tilde{B}_s^{(j)}\\
&- \frac{1}{2}\int_0^T (b^{(j)}(s,\tilde{B}_s^{(j)})+\phi^{(j)}(s))^2 ds\bigg] \Bigg\} \bigg)^{2 q_{\varepsilon}}\Bigg]^{\frac{1}{2 q_{\varepsilon}}}\\
&\leq C E_{\tilde{\mu}} \left[ \exp\left\{ 16 q_{\varepsilon}(8 q_{\varepsilon}-1)\int_0^T |b(s,\tilde{B}_s)|^2 ds\right\}\right]^{\frac{1}{8q_{\varepsilon}}} \exp\left\{2(8 q_{\varepsilon}-1) \int_0^T |\phi(s)|^2 ds\right\}.
\end{align*}

Altogether, we obtain
$$\left\vert S(Y_{t}^{i,b_{n}}-Y_{t}^{i,b})(\phi )\right\vert \leq const\cdot
E[J_{n}]^{\frac{1}{1+\varepsilon}}\cdot \exp\left\{2\left(8 \frac{1+\varepsilon}{\varepsilon}-1\right) \int_0^T |\phi(s)|^2 ds\right\}.$$
\end{proof}

\begin{lemm}
\label{squareint} Let $b_{n}:[0,T]\times \mathbb{R}^{d}\mathbb{\longrightarrow } \mathbb{R}^{d}$ be a sequence of smooth functions with compact support with $b_0:=b$ which approximate the coefficient $b:[0,T]\times \mathbb{R}^{d}\mathbb{\longrightarrow } \mathbb{R}^{d}$ in $L_p^q$. Then for any $0 \leq t \leq T$ there exists a subsequence of the corresponding strong solutions $X_{n_j,t} = Y_t^{b_{n_j}}$, $j=1,2...$, such that 
$$
Y^{b_{n_j}}_t \longrightarrow Y^b_t
$$ 
for $j\rightarrow \infty$ in $L^2(\Omega)$. In particular this implies $Y_{t}^{b}\in L^{2}(\Omega ),$ $0\leq t\leq T$.
\end{lemm}

\begin{proof}
By Corollary \ref{relcompact} we know that there exists a subsequence $Y_t^{b_{n_j}}$, $j\geq 1$, converging in $L^2(\Omega)$. Further, we need to show that $E[J_{n_j}] \rightarrow 0$ as $j\to \infty$ with $J_{n_j}$ as in (\ref{Jn}). To this end, observe that for a function $f\in L_p^q$ one has
$$E\left[ \int_0^T f(s,\tilde{B}_s)ds\right] = \int_0^T (2\pi s)^{-d/2} \int_{\R^d} f(s,z) e^{-|z|^2 / (2s)}dzds.$$

Then by using H\"{o}lder's inequality with respect to $z$ and then to $s$ we see that for any \mbox{$p',q'\in [1,\infty]$} satisfying
$$\frac{d}{p'}+\frac{2}{q'}<2,$$
we have
$$E\left[ \int_0^T f(s,\tilde{B}_s)ds\right] \leq C \|f\|_{L_{p'}^{q'}},$$
where $C$ is a constant depending on $T,d,p',q'$. Then from condition (\ref{pqcondition}), since $p,q>2$ we can find an $\delta\in [0,1)$ small enough so that $p,q>2(1+\delta)$. For these $p,q$ define $p' := \frac{p}{2(1+\delta)}\geq 1$ and $q:=\frac{q'}{2(1+\delta)}>1$ and apply the above estimate to $|f|^{2(1+\delta)}$ to obtain
\begin{align}\label{Lpqbound}
E\left[ \int_0^T |f(s,\tilde{B}_s)|^{2(1+\delta)} ds\right] \leq C \|f\|_{L_p^q}.
\end{align}

Now since $b_n^{(j)}-b^{(j)}\in L_p^q$ for every $j=1,\dots,d$ and $0<\frac{1+\varepsilon}{2}<1$ we have

$$E\left[ \left(\int_0^T (b_n^{(j)}(s,\tilde{B}_s^{(j)}) - b^{(j)}(s,\tilde{B}_s^{(j)}))^2ds\right)^{\frac{1+\varepsilon}{2}}\right] \leq E\left[ \int_0^T (b_n^{(j)}(s,\tilde{B}_s^{(j)}) - b^{(j)}(s,\tilde{B}_s^{(j)}))^2ds\right]^{\frac{1+\varepsilon}{2}}$$
which goes to zero by the above estimate (\ref{Lpqbound}) by just taking the case where $\delta=0$.

Finally, for the the second term in $E[J_{n_j}]$ we have

\begin{align*}
&E\Bigg[\bigg|\int_0^T (b^{(j)}(s,\tilde{B}_s^{(j)})^2 - b_n^{(j)}(s,\tilde{B}_s^{(j)})^2)ds\bigg|^{1+\varepsilon}\Bigg]\\
&\leq T^{\varepsilon} E\Bigg[\int_0^T (b^{(j)}(s,\tilde{B}_s^{(j)}) + b_n^{(j)}(s,\tilde{B}_s^{(j)}))^{1+\varepsilon}(b^{(j)}(s,\tilde{B}_s^{(j)}) - b_n^{(j)}(s,\tilde{B}_s^{(j)}))^{1+\varepsilon}ds\Bigg]\\
&\leq T^{\varepsilon} \int_0^T E\left[(b^{(j)}(s,\tilde{B}_s^{(j)}) + b_n^{(j)}(s,\tilde{B}_s^{(j)}))^{2(1+\varepsilon)}\right]^{1/2}E\left[(b^{(j)}(s,\tilde{B}_s^{(j)}) - b_n^{(j)}(s,\tilde{B}_s^{(j)}))^{2(1+\varepsilon)}\right]^{1/2}ds\\
&\leq T^{\varepsilon} E\left[\int_0^T (b^{(j)}(s,\tilde{B}_s^{(j)}) + b_n^{(j)}(s,\tilde{B}_s^{(j)}))^{2(1+\varepsilon)} ds\right]^{1/2}E\left[\int_0^T (b^{(j)}(s,\tilde{B}_s^{(j)}) - b_n^{(j)}(s,\tilde{B}_s^{(j)}))^{2(1+\varepsilon)} ds\right]^{1/2}.
\end{align*}
Then since $b^{(j)}+b_n {(j)}\in L_p^q$ for every $n\geq 0$ we have
$$\sup_{n\geq 0} E\left[\int_0^T (b^{(j)}(s,\tilde{B}_s^{(j)}) + b_n^{(j)}(s,\tilde{B}_s^{(j)}))^{2(1+\varepsilon)} ds\right]^{1/2} <\infty
$$
for a sufficiently small $\varepsilon \in (0,1)$ by Lemma \ref{lemma1} and
$$
E\left[\int_0^T (b^{(j)}(s,\tilde{B}_s^{(j)}) - b_n^{(j)}(s,\tilde{B}_s^{(j)}))^{2(1+\varepsilon)} ds\right]^{1/2} \to 0
$$
as $n\to \infty$ by estimate (\ref{Lpqbound}) for a sufficiently small $\varepsilon>0$.

Thus, by Lemma \ref{diffestimate}, $Y^{b_{n_j}}_t \rightarrow Y^b_t $ as $j\to \infty$ in  $(\mathcal{S})^*$. But then, by uniqueness of the limit, also $Y^{b_{n_j}}_t \rightarrow Y^b_t$ in $L^2(\Omega)$.
\end{proof}

\begin{rem} \label{rmkL2}
It follows from the above proof that $Y_t^{b_n} \to Y_t^b$ as $n\to \infty$ in $L^2(\Omega;\R^d)$ for all $t$ and $x$.
\end{rem}

In fact, Lemma \ref{squareint} enables us now to state the following "transformation property" for $Y_t^b$.

\begin{lemm}
\label{translemma} Assume that $b:[0,T]\times \mathbb{R}^{d}\mathbb{%
\longrightarrow R}^{d}$ is in $L_p^q$. Then%
\begin{equation}
\varphi ^{(i)}\left( t,Y_{t}^{b}\right) =E_{\widetilde{\mu }}\left[ \varphi
^{(i)}\left( t,\widetilde{B}_{t}\right) \mathcal{E}_{T}^{\diamond }(b)\right]
\label{transprop}
\end{equation}%
a.e. for all $0\leq t\leq T,$ $i=1,\ldots,d$ and $\varphi =(\varphi
^{(1)},\ldots,\varphi ^{(d)})$ such that $\varphi (B_{t})\in L^{2}(\Omega ;\mathbb{%
R}^{d}).$
\end{lemm}

\begin{proof}
See \cite[Lemma 16]{Pro07} or \cite{MBP06}.
\end{proof}

Using the above auxiliary results we can finally give the proof of Theorem \ref{thmainres1}.

\begin{proof}[Proof of Theorem \protect\ref{thmainres1}]
We want to use the transformation property (\ref{transprop}) of Lemma~%
\ref{translemma} to show that $Y_{t}^{b}$ is a unique strong solution of
the SDE (\ref{eqmain1}). To shorten notation we set $\int_{0}^{t}\varphi
(s,\omega )dB_{s}:=\sum_{j=1}^{d}\int_{0}^{t}\varphi ^{(j)}(s,\omega
)dB_{s}^{(j)}$ and $x=0.$ Also, let $b_n$, $n=1,2,...$, be a sequence of functions as required in Lemma \ref{squareint}.

We comment on that $Y_{\cdot }^{b}$ has a continuous modification. The
latter can be seen as follows: Since each $Y_{t}^{b_{n}}$ is a strong
solution of the SDE (\ref{eqmain1}) with respect to the drift $b_{n}$ we
obtain from Girsanov's theorem and our assumptions that%
\begin{align*}
E_{\mu }\left[ \left( Y_{t}^{i,b_{n}}-Y_{u}^{i,b_{n}}\right) ^{4}\right]
&=E_{\widetilde{\mu }}\left[ \left( \widetilde{B}_{t}^{(i)}-\widetilde{B}%
_{u}^{(i)}\right) ^{4}\mathcal{E}\left( \int_{0}^{T}b_{n}(s,\widetilde{B}%
_{s})d\widetilde{B}_{s}\right) \right] \\
&\leq const\cdot \left\vert t-u\right\vert^2
\end{align*}%
for all $0\leq u,t\leq T$, $n\geq 1$, $i=1,...,d.$ The above constant comes from the fact that $\left\{ \mathcal{E}\left( \int_{0}^{T}b_{n}(s,\widetilde{B}_{s})d\widetilde{B}_{s}\right) \right\}_{n \geq 1}$ is bounded in $L^2(\Omega;\mathbb{R}^d)$ with respect to the measure $\tilde{\mu}$, see Lemma 3.2. in \cite{KR05} or Lemma \ref{expmoments}.

By Remark \ref{rmkL2} we know that%
\begin{equation*}
Y_{t}^{b_{n}}\longrightarrow Y_{t}^{b}\text{ in }L^{2}(\Omega ;\mathbb{R}%
^{d})
\end{equation*}%
and hence we have almost sure convergence for a further subsequence, $0\leq t\leq T$. So we get that by Fatou's lemma %
\begin{equation}
E_{\mu }\left[ \left( Y_{t}^{i,b}-Y_{u}^{i,b}\right) ^{4}\right] \leq const
\cdot \left\vert t-u\right\vert^2  \label{Kolmogorov}
\end{equation}%
for all $0\leq u,t\leq T$, $i=1,...,d.$ Then Kolmogorov's lemma guarantees a
continuous modification of $Y_{t}^{b}$.

Since $\widetilde{B}_{t}$ is a weak solution of (\ref{eqmain1}) for the
drift $b(s,x)+\phi (s)$ with respect to the measure $d\mu ^{\ast }=\mathcal{E%
}\left( \int_{0}^{T}\left( b(s,\widetilde{B}_{s})+\phi (s)\right) d%
\widetilde{B}_{s}\right) d\mu $ we get that%
\begin{align*}
S(Y_{t}^{i,b})(\phi )&=E_{\widetilde{\mu }}\left[ \widetilde{B}_{t}^{(i)}%
\mathcal{E}\left( \int_{0}^{T}\left( b(s,\widetilde{B}_{s})+\phi (s)\right) d%
\widetilde{B}_{s}\right) \right] \\
& =E_{\mu ^{\ast }}\left[ \widetilde{B}_{t}^{(i)}\right] \\
& =E_{\mu ^{\ast }}\left[ \int_{0}^{t}\left( b^{(i)}(s,\widetilde{B}%
_{s})+\phi ^{(i)}(s)\right) ds\right] \\
& =\int_{0}^{t}E_{\widetilde{\mu }}\left[ b^{(i)}(s,\widetilde{B}_{s})%
\mathcal{E}\left( \int_{0}^{T}\left( b(u,\widetilde{B}_{u})+\phi (u)\right) d%
\widetilde{B}_{u}\right) \right] ds+S\left( B_{t}^{(i)}\right) (\phi ).
\end{align*}%
Thus the transformation property (\ref{transprop}) applied to $b$ yields
\begin{equation*}
S(Y_{t}^{i,b})(\phi )=S(\int_{0}^{t}b^{(i)}(u,Y_{u}^{i,b})du)(\phi
)+S(B_{t}^{(i)})(\phi ).
\end{equation*}%
Then it follows from the injectivity of the $S$-transform that 
\begin{equation*}
Y_{t}^{b}=\int_{0}^{t}b(s,Y_{s}^{b})ds+B_{t} \,.
\end{equation*}%
See Section \ref{mallcalgaus} in the Appendix.

The Malliavin differentiability of $Y_{t}^{b}$ comes from the fact that $Y_t^{i,b_n} \to Y_t^{i,b}$ in $L^2(\Omega)$ and%
\begin{equation*}
\sup_{n\geq 1} \| Y_{t}^{i,b_{n}}\|_{\mathbb{D}^{1,2}}\leq M<\infty
\end{equation*}%
for all $i=1,\ldots,d$ and $0\leq t\leq 1.$ See e.g. \cite{Nua10}.

On the other hand, using uniqueness in law, which is a consequence of Lemma \ref{lemma1} and Proposition 3.10, Ch. 5 in \cite{Kar98} we may apply, under our conditions, Girsanov's theorem to any other solution. Then the proof of Proposition \ref{explicit} (see e.g. \cite[Proposition 1]{PS91}) shows that any other solution necessarily coincides with $Y_t^b$.
\end{proof}

We conclude this section with a generalisation of Theorem \ref{thmainres1} to a class of
non-degenerate $d-$dimensional It\^{o}-diffusions$.$

\begin{thm}\label{generalsde}
Assume the time-homogeneous $\mathbb{R}^{d}-$valued SDE%
\begin{equation}
dX_{t}=b(X_{t})dt+\sigma (X_{t})dB_{t},\,\,\text{ }X_{0}=x\in \mathbb{R}^{d},%
\text{ }\,\,0\leq t\leq T,  \label{SDE}
\end{equation}%
where the coefficients $b:\mathbb{R}^{d}\longrightarrow \mathbb{R}^{d}$ and $%
\sigma :\mathbb{R}^{d}\longrightarrow \mathbb{R}^{d}\times $ $\mathbb{R}^{d}$%
are Borel measurable. Suppose that there exists a bijection $\Lambda :%
\mathbb{R}^{d}\longrightarrow \mathbb{R}^{d}$, which is twice continuously
differentiable. Let \mbox{$\Lambda _{x}:\mathbb{R}^{d}\longrightarrow L\left( 
\mathbb{R}^{d},\mathbb{R}^{d}\right)$} and $\Lambda _{xx}:\mathbb{R}%
^{d}\longrightarrow L\left( \mathbb{R}^{d}\times \mathbb{R}^{d},\mathbb{R}%
^{d}\right) $ be the corresponding derivatives of $\Lambda $ and assume that%
\begin{equation*}
\Lambda _{x}(y)\sigma (y)=id_{\mathbb{R}^{d}}\text{ for }y\text{ a.e.}
\end{equation*}%
as well as%
\begin{equation*}
\Lambda ^{-1}\text{ is Lipschitz continuous.}
\end{equation*}%
Require that the function $b_{\ast }:\mathbb{R}^{d}\longrightarrow \mathbb{R}%
^{d}$ given by 
\begin{align*}
b_{\ast }(x)&:=\Lambda _{x}\left( \Lambda ^{-1}\left( x\right) \right) 
\left[ b(\Lambda ^{-1}\left( x\right) )\right] \\
&\,\,\,+\frac{1}{2}\Lambda _{xx}\left( \Lambda ^{-1}\left( x\right) \right) \left[
\sum_{i=1}^{d}\sigma (\Lambda ^{-1}\left( x\right) )\left[ e_{i}\right]
,\sum_{i=1}^{d}\sigma (\Lambda ^{-1}\left( x\right) )\left[ e_{i}\right] %
\right]
\end{align*}%
satisfies the conditions of Theorem \ref{thmainres1}, where $e_{i},$ $%
i=1,\ldots,d$, is a basis of $\mathbb{R}^{d}.$ Then there exists a Malliavin
differentiable solution $X_{t}$ to (\ref{SDE}).
\end{thm}

\begin{proof}
The proof can be directly obtained from It\^o's Lemma. See \cite{MBP10}.
\end{proof}
 

\section{Applications}  \label{applications}

\subsection{The Bismut-Elworthy-Li formula}
\label{SubSecKol}

As an application we want to use Theorem \ref{thmainres1} to derive a Bismut-Elworthy-Li formula for solutions $v$ to the Kolmogorov equation

\begin{equation} \label{kolmogorovEq}
\frac{\partial}{\partial t} v(t,x) = \sum_{j=1}^d b_j(t,x) \frac{\partial}{\partial x_j} v(t,x) + \frac{1}{2} \sum_{i=1}^d \frac{\partial^2}{\partial x_i^2 } v(t,x)
\end{equation}
with initial condition $v(0,x)= \Phi(x)$, where $b:[0,T] \times \mathbb{R}^d \rightarrow \mathbb{R}^d$ belongs to $L_p^q$.

It is known that, see \cite{KrylRock05} or \cite{FlanLN11}, that when $\Phi $ is continuous and bounded there exists a solution to (\ref{kolmogorovEq}) given by
\begin{equation} \label{StochRepr}
v(t,x) = E[\Phi(X_t^x)] ,
\end{equation}
where $v$ is a solution to the Kolmogorov Equation (\ref{kolmogorovEq}) which is unique among all bounded solutions in the space $H_{2,p}^q$, as introduced in Theorem \ref{theoremU}, with $p,q>2$ satisfying \eqref{pqcondition}. Moreover, $\frac{\partial}{\partial x} v \in L^\infty ([0,T]\times \R^d)$.

In the sequel, we aim at finding a representation for $\frac{\partial}{\partial x} v$ without using derivatives of $\Phi$. See \cite{MMNPZ10} in the case of $b\in L^\infty ([0,T]\times \R^d)$.

\begin{thm}[Bismut-Elworthy-Li formula] \label{kolmogorovDerivative}
Assume $\Phi \in C_b(\mathbb{R}^d)$ and let $U$ be an open, bounded subset of $\mathbb{R}^d$. Then the derivative of the solution to (\ref{kolmogorovEq}) can be represented as
\begin{equation} \label{derivativeFree}
\frac{\partial}{\partial x} v(t,x) = E[ \Phi(X_t^x) \int_0^t a(s) \left(\frac{\partial}{\partial x} X_s^x\right)^\ast dB_s]^\ast
\end{equation}
for almost all $x \in U$ and all $t \in (0,T]$, where $a=a_t$ is any bounded measurable function such that $\int_0^t a_t(s)ds=1$ and where $\ast$ denotes the transposition of matrices.
\end{thm}

\begin{proof}
The proof is similar to Theorem 2 in \cite{MBP10} in the case of $b\in L^\infty ([0,T]\times \R^d)$. For the convenience of the reader we give the full proof.

Assume that $\Phi \in C^2_b(\mathbb{R}^d)$ (the general case of $\Phi \in C_b(\mathbb{R}^d)$ can be proved by approximation of $\Phi$ in relation (\ref{integralDerivative})) and let $b_n$ and $X_t^{n,x}$ be as in the previous section. If we replace $b$ by $b_n$ in (\ref{kolmogorovEq}) we have the unique solution given by 
$$
v_n(t,x)= E[\Phi(X_t^{n,x})] .
$$
By using Remark \ref{rmkL2} we see that $v_n(t,x) \rightarrow v(t,x)$ for each $t$ and $x$.

By \cite [Page 109] {Nua10} we have that 
$$
D_sX_t^{n,x} \frac{\partial}{\partial x} X_s^{n,x} = \frac{\partial}{\partial x} X_t^{n,x} ,
$$
where the above product is the usual matrix product. So it follows that
\begin{equation}
\frac{\partial}{\partial x} X_t^{n,x} =  \int_0^t a(s) D_sX_t^{n,x} \frac{\partial}{\partial x} X_s^{n,x} ds .
\end{equation}
Interchanging integration and differentiation in connection with the chain rule we find that
\begin{align*}
\frac{\partial}{\partial x} v_n(t,x) &=  E[ \Phi'(X_t^{n,x}) \frac{\partial}{\partial x} X_t^{n,x}] \\
&= E[\int_0^t a(s) D_s \Phi(X_t^{n,x}) \frac{\partial}{\partial x} X_s^{n,x} ds ] \\
&= E[ \Phi(X_t^{n,x}) \int_0^t a(s) \left(\frac{\partial}{\partial x} X_s^{n,x}\right)^\ast dB_s]^\ast , \\
\end{align*}
where we applied the chain rule and the duality formula for the Malliavin derivative to the last equality.

Choose $\varphi \in C^{\infty}_0(U)$. In what follows, we will prove that 
\begin{equation} \label{integralDerivative}
\int_{\mathbb{R}^d} \frac{\partial}{\partial x} \varphi(x) v(t,x) dx = - \int_{\mathbb{R}^d} \varphi(x) E[ \Phi(X_t^x) \int_0^t a(s) \left(\frac{\partial}{\partial x} X_s^x\right)^\ast dB_s]^\ast dx .
\end{equation}
In fact, dominated convergence combined with Remark \ref{rmkL2} gives 
\begin{align*}
\int_{\mathbb{R}^d} \frac{\partial}{\partial x} \varphi(x) v(t,x) dx & = - \lim_{n \rightarrow \infty} \int_{\mathbb{R}^d} \varphi(x) E[ \Phi(X_t^{n,x}) \int_0^t a(s) \left(\frac{\partial}{\partial x} X_s^{n,x}\right)^\ast dB_s]^\ast dx \\
&  =    - \lim_{n \rightarrow \infty} \int_{\mathbb{R}^d} \varphi(x) E[ \left( \Phi(X_t^{n,x}) - \Phi(X_t^x) \right) \int_0^t a(s) \left(\frac{\partial}{\partial x} X_s^{n,x}\right)^\ast dB_s]^\ast dx \\
& -  \lim_{n \rightarrow \infty} \int_{\mathbb{R}^d} \varphi(x) E[ \Phi(X_t^x) \int_0^t a(s) \left(\frac{\partial}{\partial x} X_s^{n,x}\right)^\ast dB_s]^\ast dx \\
 & = - \lim_{n \rightarrow \infty} i)_n -  \lim_{n \rightarrow \infty} ii)_n .\\
\end{align*}

As for the first term we get
$$
i)_n \leq \int_{\mathbb{R}^d} | \varphi(x)| \|\frac{\partial}{\partial x} \Phi \|_{\infty} \|X_t^{n,x} - X_t^x\|_{L^2(\Omega; \mathbb{R}^d)} \|a\|_{\infty} \left( \sup_{k \geq 1, s \in [0,T]} E[\| \frac{\partial}{\partial x} X_s^{k,x} \|^2_{\mathbb{R}^{d \times d}} ] \right)^{1/2} dx ,
$$
which goes to zero as $n$ tends to infinity by Lebesque dominated convergence theorem, Remark \ref{rmkL2} and Remark \ref{derivativebound}.

For the second term, $ii)_n$ since $X_t^x$ is \emph{Malliavin differentiable} and $\Phi \in C^2_b(\mathbb{R}^d)$ it follows from the Clark-Ocone formula that (see e.g. \cite{Nua10})
$$
\Phi(X_t^x) = E[\Phi(X_t^x)] + \int_0^t E[D_s \Phi(X_t^x) | \mathcal{F}_s] dB_s.
$$
So
\begin{align} 
\label{clarkoconeArgument1} ii)_n = & \int_{\mathbb{R}^d} \varphi(x) E[ \Phi(X_t^x) \int_0^t a(s) \left(\frac{\partial}{\partial x} X_s^{n,x}\right)^\ast dB_s ]^\ast dx \\
\label{clarkoconeArgument2} = & \int_{\mathbb{R}^d} \varphi(x) E[ \left( E[\Phi(X_t^x)] + \int_0^t E[D_s \Phi(X_t^x) | \mathcal{F}_s] dB_s \right) \int_0^t  a(s) \left(\frac{\partial}{\partial x} X_s^{n,x}\right)^\ast dB_s ]^\ast dx \\
\label{clarkoconeArgument3} =  &\int_0^t a(s)  \int_{\mathbb{R}^d} \varphi(x) E[ D_s \Phi(X_t^x)  \frac{\partial}{\partial x} X_s^{n,x} ]dx ds .
\end{align}
One checks by means of Lemma \ref{estimate1} that $\varphi(\cdot) D_s \Phi(X_t^{\cdot}) = \varphi(\cdot)  \Phi'(X_t^{\cdot}) D_sX_t^{\cdot}$ belongs to $L^2(\mathbb{R}^d \times \Omega ; \mathbb{R}^d)$ so that for each $s$, the function
$$
g_n(s)= \int_{\mathbb{R}^d} \varphi(x) E[ D_s \Phi(X_t^x) \frac{\partial}{\partial x} X_s^{n,x} ] dx
$$
converges to $\int_{\mathbb{R}^d} \varphi(x) E[ D_s \Phi(X_t^x) \frac{\partial}{\partial x} X_s^x ] dx$ by the weak convergence of $\frac{\partial}{\partial x} X_s^{n,x}$ in $L^2([0,T]\times U\times \Omega)$ for a subsequence in virtue of Remark \ref{derivativebound}. Further,
\begin{align*}
|g_n(s)| & \leq  \int_{\mathbb{R}^d} |\varphi(x)| \| D_s \Phi(X_t^x)\|_{L^2(\Omega; \mathbb{R}^d)} \| \frac{\partial}{\partial x} X_s^{n,x} \|_{L^2(\Omega; \mathbb{R}^d)} dx \\
 & \leq \sup_{y \in \mathbb{R}^d, \, u \leq t, \, k \in \mathbb{N}} \| D_u \Phi(X_t^y)\|_{L^2(\Omega; \mathbb{R}^d)} \| \frac{\partial}{\partial x} X_u^{k,y} \|_{L^2(\Omega; \mathbb{R}^d)} \int_{\mathbb{R}^d} |\varphi(x)| dx \\
\end{align*}
so that Lebesgue's dominated convergence theorem gives
$$
\lim_{n \rightarrow \infty} ii)_n =  \int_0^t a(s) \int_{\mathbb{R}^d} \varphi(x) E[ D_s \Phi(X_t^x) \frac{\partial}{\partial x} X_s^x ] dx ds .
$$
By reversing equations (\ref{clarkoconeArgument1}), (\ref{clarkoconeArgument2}) and (\ref{clarkoconeArgument3}) with $\frac{\partial}{\partial x} X_s^x$ in place of $\frac{\partial}{\partial x} X_s^{n,x}$ we obtain the result.
\end{proof}


\appendix
\section{Framework}\label{mallcalgaus}
 
In this appendix we collect some facts from Gaussian white noise analysis and
Malliavin calculus, which we shall use in Section \ref{main results} to construct strong solutions of SDE's. See \cite{HKPS93, Oba94, Kuo96} for
more information on white noise theory. As for Malliavin calculus the reader
may consult \cite{Nua10, Mall78, Mall97, DOP08}.

\subsection{Basic Facts of Gaussian White Noise Theory}\label{WN}

A crucial step in our proof for the constuction of strong solutions (see Section 3) relies on a generalised stochastic process in the Hida distribution space which is shown to be a SDE solution. Let us first recall the definition of this space which is due to T. Hida (see \cite{HKPS93}).

From now on we fix a time horizon $0<T<\infty .$ Let $A$ be a (positive) self-adjoint operator on $L^{2}([0,T])$ with $Spec(A)>1$. Require
that $A^{-r}$ is of Hilbert-Schmidt type for some $r>0$ and let $\{e_{j}\}_{j\geq 0}$ be a complete orthonormal basis of $L^{2}([0,T])$ in $%
Dom(A)$ and let $\lambda _{j}>0$, $j\geq 0$ be the eigenvalues of $A$ such
that%
\begin{equation*}
1<\lambda _{0}\leq \lambda _{1}\leq ...\longrightarrow \infty .
\end{equation*}%
Suppose that each basis element $e_{j}$ is a continuous function on $%
[0,T]$. Further let $O_{\lambda },\lambda \in \Gamma$, be an open covering
of $[0,T]$ such that%
\begin{equation*}
\sup_{j\geq 0}\lambda _{j}^{-\alpha (\lambda )}\sup_{t\in O_{\lambda
}}\left| e_{j}(t)\right| <\infty
\end{equation*}%
for $\alpha (\lambda )\geq 0.$

In the sequel let $\mathcal{S}([0,T])$ be the standard countably
Hilbertian space constructed from $(L^{2}([0,T]),A)$. See \cite{Oba94}. Then 
$\mathcal{S}([0,T])$ is a nuclear subspace of $L^{2}([0,T])$. The topological dual of $\mathcal{S}([0,T])$ is denoted by $\mathcal{S}^{\prime }([0,T])$. Then the Bochner-Minlos theorem entails the existence of a unique probability measure $\pi $ on $%
\mathcal{B}(\mathcal{S}^{\prime }([0,T]))$ (Borel $\sigma -$algebra of $\mathcal{S}^{\prime }([0,T])$) such that%
\begin{equation*}
\int_{\mathcal{S}^{\prime }([0,T])}e^{i\left\langle \omega ,\phi
\right\rangle }\pi (d\omega )=e^{-\frac{1}{2}\left\Vert \phi \right\Vert
_{L^{2}([0,T])}^{2}}
\end{equation*}%
for all $\phi \in \mathcal{S}([0,T]),$ where $\left\langle \omega
,\phi \right\rangle $ stands for the action of $\omega \in \mathcal{S}^{\shortmid
}([0,T])$ on $\phi \in \mathcal{S}([0,T]).$ Define
\begin{equation*}
\Omega _{i}=\mathcal{S}^{\prime }([0,T])\,,\ \ \mathcal{F}_{i}=\mathcal{B}%
(\mathcal{S}^{\prime }([0,T]))\,,\ \ \mu _{i}=\pi \,,
\end{equation*}%
for $i=1,\ldots,d$. Then the product measure%
\begin{equation}
\mu =\underset{i=1}{\overset{d}\times }\mu _{i}  \label{mu}
\end{equation}%
on the measurable space%
\begin{equation}
\left( \Omega ,\mathcal{F}\right) :=\left( \underset{i=1}{\overset{d}\prod}\Omega
_{i},\underset{i=1}{\overset{d}\otimes}\mathcal{F}_{i}\right)  \label{Meas}
\end{equation}%
is called $d$-\emph{dimensional white noise probability measure}.

Consider the Dol\'{e}ans-Dade exponential 
\begin{equation*}
\widetilde{e}(\phi ,\omega )=\exp \left( \left\langle \omega ,\phi
\right\rangle -\frac{1}{2}\left\| \phi \right\| _{L^{2}([0,T];\mathbb{R%
}^{d})}^{2}\right) ,
\end{equation*}%
for $\omega =(\omega _{1},\ldots,\omega _{d})\in (\mathcal{S}^{\prime
}([0,T]))^{d}$ and $\phi =(\phi ^{(1)},\ldots,\phi ^{(d)})\in (\mathcal{S}%
([0,T]))^{d}$, where \mbox{$\left\langle \omega ,\phi \right\rangle
:=\sum_{i=1}^{d}\left\langle \omega _{i},\phi _{i}\right\rangle$}.

Now let $\left( (\mathcal{S}([0,T]))^{d}\right) ^{\widehat{%
\otimes }n}$ be the $n-$th completed symmetric tensor product of $(\mathcal{S%
}([0,T]))^{d}$ with itself. One checks that $\widetilde{e}(\phi ,\omega )$
is holomorphic in $\phi $ around zero. Hence, there exist generalised Hermite
polynomials $H_{n}(\omega )\in \left( \left( (\mathcal{S}([0,T]))^{d}\right)
^{\widehat{\otimes }n}\right) ^{\prime }$ such that 
\begin{equation}
\widetilde{e}(\phi ,\omega )=\sum_{n\geq 0}\frac{1}{n!}\left\langle
H_{n}(\omega ),\phi ^{\otimes n}\right\rangle   \label{powerexpansion}
\end{equation}%
for $\phi $ in a certain neighbourhood of zero in $(\mathcal{S}([0,T]))^{d}.$
One proves that
\begin{equation}
\left\{ \left\langle H_{n}(\omega ),\phi ^{(n)}\right\rangle :\phi ^{(n)}\in
\left( (\mathcal{S}([0,T]))^{d}\right) ^{\widehat{\otimes }n},\text{ }n\in 
\mathbb{N}_{0}\right\} 
\end{equation}%
is a total set of $L^{2}(\Omega ).$ Further, it can be shown that the generalised Hermite polynomials satisfy the orthogonality relation 
\begin{equation}
\int_{\mathcal{S}^{\prime }}\left\langle H_{n}(\omega ),\phi
^{(n)}\right\rangle \left\langle H_{m}(\omega ),\psi ^{(m)}\right\rangle \mu
(d\omega )=\delta _{n,m}n!\left( \phi ^{(n)},\psi ^{(n)}\right)
_{L^{2}([0,T]^{n};(\mathbb{R}^{d})^{\otimes n})}  \label{ortho}
\end{equation}%
for all $n,m\in \mathbb{N}_{0}$, $\phi ^{(n)}\in \left( (\mathcal{S}%
([0,T]))^{d}\right) ^{\widehat{\otimes }n},$ $\psi ^{(m)}\in \left( (%
\mathcal{S}([0,T]))^{d}\right) ^{\widehat{\otimes }m}$ where

\begin{equation*}
\delta _{n,m}=\left\{ 
\begin{array}{cc}
1 & \text{if }n=m \\ 
0 & \text{else}%
\end{array}%
\right. .
\end{equation*}%
Denote by $\widehat{L}^{2}([0,T]^{n};(\mathbb{R}^{d})^{\otimes n})$ the
space of square integrable symmetric functions $f(x_{1},\ldots,x_{n})$ with
values in $(\mathbb{R}^{d})^{\otimes n}.$ Then it follows from relation (%
\ref{ortho}) that the mappings 
\begin{equation*}
\phi ^{(n)}\longmapsto \left\langle H_{n}(\omega ),\phi ^{(n)}\right\rangle
\end{equation*}
from $\left( \mathcal{S([}0,T]\mathcal{)}^{d}\right) ^{\widehat{\otimes }n}$
to $L^{2}(\Omega )$ have unique continuous extensions%
\begin{equation*}
I_{n}:\widehat{L}^{2}([0,T]^{n};(\mathbb{R}^{d})^{\otimes n})\longrightarrow
L^{2}(\Omega )
\end{equation*}%
for all $n\in \mathbb{N}$. These extensions $I_{n}(\phi ^{(n)})$ can be identified as $n$-fold iterated It\^{o} integrals of $\phi ^{(n)}\in \widehat{L}%
^{2}([0,T]^{n};(\mathbb{R}^{d})^{\otimes n})$ with respect to a $d-$%
dimensional Wiener process%
\begin{equation}
B_{t}=\left( B_{t}^{(1)},\ldots,B_{t}^{(d)}\right)  \label{Bm}
\end{equation}%
on the white noise space 
\begin{equation}
\left( \Omega ,\mathcal{F},\mu \right) \,.
\end{equation}%
We mention that square integrable functionals of $B_{t}$ admit a Wiener-It%
\^{o} chaos representation which can be regarded as an infinite-dimensional
Taylor expansion, that is%
\begin{equation}
L^{2}(\Omega )=\bigoplus_{n\geq 0} I_{n}(\widehat{L}^{2}([0,T]^{n};(\mathbb{R}%
^{d})^{\otimes n})).  \label{chaosrepr}
\end{equation}

The definition of the Hida stochastic test function and distribution space is based on the Wiener-It\^{o} chaos decomposition (\ref{chaosrepr}): Set
\begin{equation}
A^{d}:=(A,\ldots,A)\, .  \label{operator}
\end{equation}
Using a second quantisation argument, the Hida stochastic test function space $(\mathcal{S})$ is defined as the space of all $f=\sum_{n\geq 0}\left\langle H_{n}(\cdot ),\phi
^{(n)}\right\rangle \in L^{2}(\Omega )$ such that%
\begin{equation}
\left\Vert f\right\Vert _{0,p}^{2}:=\sum_{n\geq 0}n!\left\Vert \left(
(A^{d})^{\otimes n}\right) ^{p}\phi ^{(n)}\right\Vert _{L^{2}([0,T]^{n};(%
\mathbb{R}^{d})^{\otimes n})}^{2}<\infty  \label{Hidaseminorm}
\end{equation}%
for all $p\geq 0$. In fact, the space $(\mathcal{S})$ is a nuclear
Fr\'{e}chet algebra with respect to multiplication of functions and its
topology is induced by the seminorms $\left\Vert \cdot \right\Vert _{0,p},$ $%
p\geq 0.$ Further one shows that%
\begin{equation}
\widetilde{e}(\phi ,\omega )\in (\mathcal{S})  \label{Doleans}
\end{equation}%
for all $\phi \in (\mathcal{S}([0,T]))^{d}.$

On the other hand, the topological dual of $(\mathcal{S})$, denoted by $(\mathcal{S})^{\ast
}$, is called \emph{Hida stochastic distribution space}. Using these definitions we ontain the Gel'fand triple%
\begin{equation*}
(\mathcal{S})\hookrightarrow L^{2}(\Omega )\hookrightarrow (\mathcal{S})^{\ast
}.
\end{equation*}%
It turns out that the \emph{white noise} of the coordinates of the $%
d-$dimensional Wiener process $B_{t},$ that is the time derivatives%
\begin{equation}
W_{t}^{i}:=\frac{d}{dt}B_{t}^{i},\text{ }i=1,\ldots,d\,,  \label{whitenoise}
\end{equation}%
belong to $(\mathcal{S})^{\ast }$.

We also recall the definition of the $S$-transform. See \cite{PS91}. The $S-$transform of a $\Phi
\in (\mathcal{S})^{\ast }$, denoted by $S(\Phi ),$ is defined by the dual
pairing 
\begin{equation}
S(\Phi )(\phi )=\left\langle \Phi ,\widetilde{e}(\phi ,\omega )\right\rangle
\label{Stransform}
\end{equation}%
for $\phi \in (\mathcal{S}_{\mathbb{C}}([0,T]))^{d}$. Here $\mathcal{S}_{%
\mathbb{C}}([0,T])$ the complexification of $\mathcal{S}([0,T])$. The $S-$transform is a monomorphism from $(\mathcal{S})^{\ast }$ to $%
\mathbb{C}$. In particular, if 
\begin{equation*}
S(\Phi )=S(\Psi )\text{ for }\Phi ,\Psi \in (\mathcal{S})^{\ast }
\end{equation*}%
then%
\begin{equation*}
\Phi =\Psi .
\end{equation*}%
As an example one finds that
\begin{equation}
S(W_{t}^{i})(\phi )=\phi ^{i}(t),\text{ }i=1,...,d\text{ }
\label{StransfvonW}
\end{equation}%
for $\phi =(\phi ^{(1)},\ldots,\phi ^{(d)})\in (\mathcal{S}_{\mathbb{C}%
}([0,T]))^{d}.$

Finally, we recall the concept of the \textsl{Wick} or \textsl{%
Wick-Grassmann product}. The Wick product defines a tensor algebra
multiplication on the Fock space and is introduced as follows: The Wick
product of two distributions $\Phi ,\Psi \in (\mathcal{S})^{\ast }$, denoted
by $\Phi \diamond \Psi ,$ is the unique element in $(\mathcal{S})^{\ast }$
such that%
\begin{equation}
S(\Phi \diamond \Psi )(\phi )=S(\Phi )(\phi )S(\Psi )(\phi )  \label{Sprop}
\end{equation}%
for all $\phi \in (\mathcal{S}_{\mathbb{C}}([0,T]))^{d}.$ As an example, we
get
\begin{equation}
\left\langle H_{n}(\omega ),\phi ^{(n)}\right\rangle \diamond \left\langle
H_{m}(\omega ),\psi ^{(m)}\right\rangle =\left\langle H_{n+m}(\omega ),\phi
^{(n)}\widehat{\otimes }\psi ^{(m)}\right\rangle  \label{Wick}
\end{equation}%
for $\phi ^{(n)}\in \left( (\mathcal{S}([0,T]))^{d}\right) ^{\widehat{%
\otimes }n}$ and $\psi ^{(m)}\in \left( (\mathcal{S}([0,T]))^{d}\right) ^{%
\widehat{\otimes }m}.$ The latter in connection with (\ref{powerexpansion})
implies that%
\begin{equation}
\widetilde{e}(\phi ,\omega )=\exp ^{\diamond }(\left\langle \omega ,\phi
\right\rangle )  \label{Wickident}
\end{equation}%
for $\phi \in (\mathcal{S}([0,T]))^{d}.$ Here the Wick exponential $\exp
^{\diamond }(X)$ of a $X\in (\mathcal{S})^{\ast }$ is defined as 
\begin{equation}
\exp ^{\diamond }(X)=\sum_{n\geq 0}\frac{1}{n!}X^{\diamond n},
\label{wickexp}
\end{equation}%
where $X^{\diamond n}=X\diamond \ldots \diamond X,$ provided that the sum on the right hand
side converges in $(\mathcal{S})^{\ast }$.

\subsection{Basic elements of Malliavin Calculus}\label{somedefmalcal}

In this section we pass in review some basic definitions from Malliavin calculus.

For convenience we consider the case $d=1$. Let $F\in L^{2}(\Omega). $ Then we know from (\ref{chaosrepr}) that 
\begin{equation}
F=\sum_{n\geq 0}\left\langle H_{n}(\cdot ),\phi ^{(n)}\right\rangle
\label{ch}
\end{equation}%
for unique $\phi ^{(n)}\in \widehat{L}^{2}([0,T]^{n}).$ Suppose that 
\begin{equation}
\sum_{n\geq 1}nn!\left\Vert \phi ^{(n)}\right\Vert
_{L^{2}([0,T]^{n})}^{2}<\infty \,.  \label{D1,2}
\end{equation}%
Then the \emph{Malliavin derivative} $D_{t}$ of $F$ in the direction of $B_{t}$
can be defined as
\begin{equation}
D_{t}F=\sum_{n\geq 1}n\left\langle H_{n-1}(\cdot ),\phi ^{(n)}(\cdot
,t)\right\rangle .  \label{Dt}
\end{equation}%
We denote by $\mathbb{D}^{1,2}$ the space of all $F\in L^{2}(\Omega )$ such that (\ref{D1,2}) holds. The Malliavin
derivative $D_{\cdot \text{ }}$ is a linear operator from $\mathbb{D}^{1,2}$
to $L^{2}([0,T] \times \Omega )$. We mention that $\mathbb{D}^{1,2}$ is a Hilbert space with the norm 
$\left\Vert \cdot \right\Vert _{1,2}$ given by%
\begin{equation}
\left\Vert F\right\Vert _{1,2}^{2}:=\left\Vert F\right\Vert _{L^{2}(\Omega, \mu
)}^{2}+\left\Vert D_{\cdot }F\right\Vert _{L^{2}([0,T]\times \Omega ,\lambda
\times \mu )}^{2}.  \label{norm1,2}
\end{equation}%
We get the following chain of continuous inclusions:%
\begin{equation}
(\mathcal{S})\hookrightarrow \mathbb{D}^{1,2}\hookrightarrow L^{2}(\Omega)\hookrightarrow \mathbb{D}^{-1,2}\hookrightarrow (\mathcal{S})^{\ast },
\label{Inclusion}
\end{equation}%
where $\mathbb{D}^{-1,2}$ is the dual of $\mathbb{D}^{1,2}.$

\section{Technical results} \label{App1}  

We give a list if technical results needed for the proofs of Section \ref{main results} and \ref{applications}.

\begin{lemm}\label{expmoments}
Let $\{f_n\}_{n\geq 0}$ be a bounded sequence of functions in $L_{p}^q$. Then, for every $k\in \R$
$$\sup_{x\in \R^d}\sup_{n\geq 0}E\left[ \exp\left\{k\int_0^T |f_n(s,x+B_s)|^2 ds\right\} \right]<\infty.$$
In particular, there exists a weak solution to SDE (\ref{eqmain1}).
\end{lemm}
\begin{proof}
See \cite[Lemma 3.2]{KR05}
\end{proof}

\begin{lemm}\label{lemma1}
Let $\{f_n\}_{n\geq 0}$ a sequence of elements in $L^{p,q}$ that converges to some $f\in L^{p,q}$. Then there exists $\varepsilon>1$ such that
\begin{align}
\sup_{n\geq 0}E\left[ \int_0^T \|f_n(s,\phi_s^n)\|^{2\varepsilon} ds \right] <\infty.
\end{align}
Here $\phi_s^n: x\mapsto X_t^{x,n}$ denotes the stochastic flow associated to the solution of the SDE (\ref{eqmain1}) with drift coefficient $b_n\in \mathcal{C}_b^\infty (\R^d)$.
\end{lemm}
\begin{proof}
See \cite[Lemma 15]{FedFlan10}.
\end{proof}

We also need the following crucial lemma, which can be found in \cite{FedFlan12}, Lemma 3.4.

\begin{lemm}\label{lemma2}
Let $U_n$ be the solution of the PDE (\ref{PDE(f)}) with $\Phi=b=b_n \in \mathcal{C}_b^\infty (\R^n)$. Let $X_t^{x,n}$ be the solution of the SDE (\ref{eqmain1}) with drift coefficient $b_n\in \mathcal{C}_b^\infty(\R^d)$. Then the following holds true
\begin{itemize}
\item[(i)] For each $r>0$ there exists a function $f$ with $\lim_n f(n) = 0$ such that
$$\sup_{x\in B_r} \sup_{t\in[0,T]} \|U_n(t,x) - U(t,x)\| \leq f(n)$$
and
$$\sup_{x\in B_r} \sup_{t\in[0,T]} \|\nabla U_n(t,x) - \nabla U(t,x)\| \leq f(n)$$
\item[(ii)] There exists a $\lambda \in \R$ for which $\displaystyle \sup_{\substack{t\in [0,T] \\ x\in \R^d}} \|\nabla U_n(t,x)\| \leq \frac{1}{2}$.
\item[(iii)] $\displaystyle \sup_{n\geq 0} \|\Delta U_n(t,x)\|_{L^{p,q}} <\infty$.
\item[(iv)] As a consequence of the boundedness of $U_n$ and $\nabla U_n$ we have
$$\sup_{t\in [0,T]} E\left[ \| \gamma_t^n(x) \|^a\right]\leq C \left( 1 + |x|^a\right).$$
\end{itemize}
\end{lemm}

The following lemma gives a bound for the derivative of the inverse of the family of diffeomorphisms $\gamma_t$. See \cite{FedFlan12}, Lemma 3.5 for its proof.

\begin{lemm}\label{lemma3}
Let $\gamma_{t,n}:\R^d \rightarrow \R^d$ be the $C^1$-diffeomorphisms defined as $\gamma_{t,n}(x) := x + U_n(t,x)$ for $x\in \R^d$ associated to $X_t^{x,n}$ the solution of SDE (\ref{eqmain1}) with drift coefficient $b_n\in C_b^\infty(\R^d)$. Then
$$\displaystyle \sup_{n\geq 0} \sup_{t\in [0,T]} \|\nabla \gamma_{t,n}^{-1} \|_{C(\R^d)} \leq 2.$$
\end{lemm}

The next result was shown in \cite{FedFlan11}, Corollary 13.
\begin{lemm}\label{lemma4}
Let $V_t^n$ be the process defined in (\ref{processV}). Then for every $\alpha \in \R$
$$\sup_{n\geq 0} E\left[e^{\alpha V_T^n}\right] \leq C.$$
Observe that the same estimate holds for any $t\in[0,T]$ since $V_t^n$ is an increasing process.
\end{lemm}

\section{} \label{App2}

The following result which is due to \cite[Theorem 1] {DPMN92} gives a compactness criterion for subsets of $L^{2}(\Omega ;\mathbb{R}^{d})$ using Malliavin calculus.

\begin{thm}
\label{MCompactness}Let $\left\{ \left( \Omega ,\mathcal{A},P\right)
;H\right\} $ be a Gaussian probability space, that is $\left( \Omega ,%
\mathcal{A},P\right) $ is a probability space and $H$ a separable closed
subspace of Gaussian random variables of $L^{2}(\Omega )$, which generate
the $\sigma $-field $\mathcal{A}$. Denote by $\mathbf{D}$ the derivative
operator acting on elementary smooth random variables in the sense that%
\begin{equation*}
\mathbf{D}(f(h_{1},\ldots,h_{n}))=\sum_{i=1}^{n}\partial
_{i}f(h_{1},\ldots,h_{n})h_{i},\text{ }h_{i}\in H,f\in C_{b}^{\infty }(\mathbb{R%
}^{n}).
\end{equation*}%
Further let $\mathbf{D}_{1,2}$ be the closure of the family of elementary
smooth random variables with respect to the norm%
\begin{equation*}
\left\Vert F\right\Vert _{1,2}:=\left\Vert F\right\Vert _{L^{2}(\Omega
)}+\left\Vert \mathbf{D}F\right\Vert _{L^{2}(\Omega ;H)}.
\end{equation*}%
Assume that $C$ is a self-adjoint compact operator on $H$ with dense image.
Then for any $c>0$ the set%
\begin{equation*}
\mathcal{G}=\left\{ G\in \mathbf{D}_{1,2}:\left\Vert G\right\Vert
_{L^{2}(\Omega )}+\left\Vert C^{-1} \mathbf{D} \,G\right\Vert _{L^{2}(\Omega ;H)}\leq
c\right\}
\end{equation*}%
is relatively compact in $L^{2}(\Omega )$.
\end{thm}

A useful bound in connection with Theorem \ref{MCompactness}, based on fractional Sobolev spaces is the following (see \cite{DPMN92}):

\begin{lemm}
\label{DaPMN} Let $v_{s},s\geq 0$ be the Haar basis of $L^{2}([0,T])$. For
any $0<\alpha <1/2$ define the operator $A_{\alpha }$ on $L^{2}([0,T])$ by%
\begin{equation*}
A_{\alpha }v_{s}=2^{k\alpha }v_{s}\text{, if }s=2^{k}+j\text{ }
\end{equation*}%
for $k\geq 0,0\leq j\leq 2^{k}$ and%
\begin{equation*}
A_{\alpha }T=T.
\end{equation*}%
Then for all $\beta $ with $\alpha <\beta <(1/2),$ there exists a constant $%
c_{1}$ such that%
\begin{equation*}
\left\Vert A_{\alpha }f\right\Vert \leq c_{1}\left\{ \left\Vert f\right\Vert
_{L^{2}([0,T])}+\left( \int_{0}^{T}\int_{0}^{T}\frac{\left\vert
f(t)-f(t^{\prime })\right\vert ^{2}}{\left\vert t-t^{\prime }\right\vert
^{1+2\beta }}dt\,dt^{\prime }\right) ^{1/2}\right\} .
\end{equation*}
\end{lemm}

A direct consequence of Theorem \ref{MCompactness} and Lemma \ref{DaPMN} is now the following compactness criterion which is essential for the proof of Corollary \ref{relcompact}.

\begin{cor} \label{compactcrit}
Let a sequence of $\mathcal{F}_T$-measurable random variables $X_n\in\mathbb{D}_{1,2}$, $n=1,2...$, be such that there exist constants $\alpha > 0$ and $C>0$ with
$$
\sup_n E[|X_n|^2] \leq C ,
$$
$$
\sup_n E \left[ \| D_t X_n - D_{t'} X_n \|^2 \right] \leq C |t -t'|^{\alpha}
$$
for $0 \leq t' \leq t \leq T$ and
$$
\sup_n\sup_{0 \leq t \leq T} E \left[ \| D_t X_n \|^2 \right] \leq C \,.
$$
Then the sequence $X_n$, $n=1,2...$, is relatively compact in $L^{2}(\Omega )$.
\end{cor}

\end{document}